\documentclass{article}
\usepackage{fullpage}
\usepackage[utf8]{inputenc} 
\usepackage[T1]{fontenc}    
\usepackage{hyperref}       
\usepackage{url}            
\usepackage{booktabs}       
\usepackage{amsfonts}       
\usepackage{nicefrac}       
\usepackage{microtype}      

\title{Exploring Algorithmic Fairness in \\ Robust Graph Covering Problems}

\author{
  Aida Rahmattalabi, Phebe Vayanos, Anthony Fulginiti, \\ Eric Rice, Bryan Wilder, 
  Amulya Yadav, Milind Tambe}

\usepackage{subfig}
\usepackage{graphicx}
\usepackage{color}
\usepackage[ruled,vlined,linesnumbered,noresetcount]{algorithm2e}
\usepackage{bbm}
\usepackage[bbgreekl]{mathbbol}
\usepackage{amsfonts}
\usepackage{amsmath,bm}
\usepackage{amsthm,amssymb}
\everymath{\displaystyle}

\newcommand{\lb}{\boldsymbol{\ell}}

\newtheoremstyle{thm-sf}{}{}{\itshape}{}{\bfseries}{.}{ }{}
\theoremstyle{thm-sf}

\newtheorem{assumption}{Assumption}
\newtheorem{definition}{Definition}

\newtheorem{theorem}{Theorem}

\newtheorem{lemma}{Lemma}

\newtheorem{proposition}{Proposition}

\newcommand{\cmmnt}[1]{}
\usepackage{color}

\newcommand{\newpv}[1]{{\color{black} #1}}

\newcommand{\newar}[1]{{\color{black} #1}}

\usepackage{enumitem}   
\usepackage{multirow}
\usepackage{xcolor,colortbl}

\definecolor{Gray}{gray}{0.85}
\definecolor{LightCyan}{rgb}{0.88,1,1}
\newcolumntype{a}{>{\columncolor{Gray}}c}
\newcolumntype{b}{>{\columncolor{white}}c}
\DeclareMathOperator{\EX}{\mathbb{E}}
\usepackage{multirow}
\usepackage{xcolor,colortbl}

\usepackage{wrapfig}
\date{}
\begin{document}

\maketitle
\begin{abstract}
Fueled by algorithmic advances, AI algorithms are increasingly being deployed in settings subject to unanticipated challenges with complex social effects. {Motivated by real-world deployment of AI driven}, social-network based suicide prevention and landslide risk management interventions, this paper focuses on robust graph covering problems subject to group fairness constraints. We show that, in the absence of fairness constraints, state-of-the-art algorithms for the robust {graph} covering problem result in biased node coverage: they tend to discriminate individuals (nodes) based on membership in traditionally marginalized groups. To mitigate this issue, we propose a novel formulation of the robust graph covering problem with group fairness constraints and a tractable approximation scheme applicable to real-world instances. We provide a formal analysis of the price of group fairness (PoF) for this problem, where we show that uncertainty can lead to greater PoF. We demonstrate the effectiveness of our approach on several real-world social networks. Our method yields competitive node coverage while significantly improving group fairness relative to state-of-the-art methods.
\end{abstract}

\section{Introduction}

\textbf{Motivation.} This paper considers the problem of selecting a subset of nodes (which we refer to as `monitors') in a graph that can `cover' their adjacent nodes. We are mainly motivated by settings where monitors are subject to failure and we seek to maximize worst-case node coverage. We refer to this problem as the \emph{robust graph covering}. This problem finds applications in several critical real-world domains, especially in the context of optimizing social interventions on vulnerable populations. Consider for example the problem of designing \emph{Gatekeeper training interventions for suicide prevention}, wherein a small number of individuals can be trained to identify warning signs of suicide among their peers~\cite{isaac2009gatekeeper}. A similar problem arises in the context of \emph{disaster risk management in remote communities} wherein a moderate number of individuals are recruited in advance and trained to watch out for others in case of natural hazards (e.g., in the event of a landslide~\cite{landslide}). Previous research has shown that social intervention programs of this sort hold great promise \cite{isaac2009gatekeeper,landslide}. Unfortunately, in these real-world domains, intervention agencies often have very limited resources, e.g., moderate number of social workers to conduct the intervention, small amount of funding to cover the cost of training. This makes it essential to target the right set of monitors to cover a maximum number of nodes in the network. Further, in these interventions, the performance and availability of individuals (monitors) is \emph{unknown} and \emph{unpredictable}. At the same time, robustness is desired to guarantee high coverage even in worst-case settings to make the approach suitable for deployment in the open world.

Robust graph covering problems similar to the one we consider here have been studied in the literature, for example, see~\cite{bogunovic2017robust,tzoumas2017resilient}. Yet, a major consideration distinguishes our problem from previous work: namely, the need for fairness. Indeed, when deploying interventions in the open world (especially in sensitive domains impacting life and death like the ones that motivate this work), care must be taken to ensure that algorithms do not discriminate among people with respect to protected characteristics such as race, ethnicity, disability, etc. In other words, we need to ensure that independently of their group, individuals have a high chance of being covered, a notion we refer to as \textit{group fairness}. 

\begin{table}[t] 
  \small
  \centering
  \begin{tabular}{r*{7}{c}}
    \toprule
    \centering
    \multirow{2}{*}{Network Name} & \multirow{2}{*}{Network Size} &  \multicolumn{5}{c}{Worst-case coverage of individuals by racial group (\%)} \\
    \cmidrule(lr){3-7}
     &  & White & Black & Hispanic & Mixed  & Other  \cmmnt{& Coverage} \\
    \midrule
    \texttt{SPY1} & 95 & 70 & \textbf{36} & -- & 86 & 94 \\
    \texttt{SPY2} & 117 & 78 & -- & \textbf{42} & 76 & 67 \\
    \texttt{SPY3} & 118 & 88 & -- & \textbf{33} & 95 & 69 \\
    \texttt{MFP1} & 165  & 96 & 77 & 69 & 73 & \textbf{28} \cmmnt{& 105} \\
    \texttt{MFP2} & 182  & \textbf{44} & 85 & 70 & 77 & 72 \cmmnt{&126} \\
  \bottomrule \\
  \end{tabular}
  \caption{Racial discrimination in node coverage resulting from applying the algorithm in \cite{tzoumas2017resilient} on real-world social networks from two homeless drop-in centers in Los Angeles, CA~\cite{barman2016sociometric}, when 1/3 of nodes (individuals) can be selected as monitors, out of which at most 10\% will fail. The numbers correspond to the worst-case percentage of covered nodes across all monitor availability scenarios.}
  \label{table:DC-Greedy-Discrimination}
\end{table}

To motivate our approach, consider deploying in the open world a state-of-the art algorithm for robust graph covering (which does not incorporate fairness considerations). Specifically, we apply the solutions provided by the algorithm from \cite{tzoumas2017resilient} on five real-world social networks. 
The results are summarized in Table~\ref{table:DC-Greedy-Discrimination} where, for each network, we report its size and the worst-case coverage by racial group. In all instances, there is significant disparity in coverage across racial groups. As an example, in network \texttt{SPY1} 36\% of Black individuals are covered in the worst-case compared to 70\% (resp.\ 86\%) of White (resp.\ Mixed race) individuals. Thus, when maximizing coverage without fairness, (near-)optimal interventions end up mirroring any differences in degree of connectedness of different groups. In particular, well-connected groups at the center of the network are more likely to be covered (protected). Motivated by the desire to support those that are the less well off, we employ ideas from \emph{maximin fairness} to improve coverage of those groups that are least likely to be protected.

\textbf{Proposed Approach and Contributions.} We investigate the \emph{robust graph covering problem with fairness constraints}. Formally, given a social network, where each node belongs to a group, we consider the problem of selecting a subset of~$I$ nodes (monitors), when at most~$J$ of them may fail. When a node is chosen as a monitor and does not fail, all of its neighbors are said to be `covered' and we use the term `coverage' to refer to the total number of covered nodes. Our objective is to maximize worst-case coverage when any $J$ nodes may fail, while ensuring fairness in coverage across groups. We adopt maximin fairness from the Rawlsian theory of justice~\cite{rawls2009theory} as our fairness criterion: we aim to maximize the utility of the groups that are worse-off. To the best of our knowledge, ours is the first paper enforcing fairness constraints in the context of graph covering subject to node failure. 

We make the following contributions: \emph{(i)} We achieve maximin group fairness by incorporating constraints inside a robust optimization model, wherein we require that at least a fraction~$W$ of each group is covered, in the worst-case; \emph{(ii)} We propose a novel two-stage robust optimization formulation of the problem for which near-optimal conservative approximations can be obtained as a moderately-sized mixed-integer linear program (MILP). By leveraging the decomposable structure of the resulting MILP, we propose a Benders' decomposition algorithm augmented with symmetry breaking to solve practical problem sizes; \emph{(iii)} We present the first study of price of group fairness (PoF), i.e., the loss in coverage due to fairness constraints in the graph covering problem subject to node failure. We provide upper bounds on the PoF for Stochastic Block Model networks, a widely studied model of networks with community structure; \emph{(iv)}~Finally, we demonstrate the effectiveness of our approach on several real-world social networks of homeless youth. Our method yields competitive node coverage while significantly improving group fairness relative to state-of-the-art methods. 

\textbf{Related Work.} Our paper relates to three streams of literature which we review.

\textit{Algorithmic Fairness.} With increase in deployments of AI, OR, and ML algorithms for decision and policy-making in the open world has come increased interest in algorithmic fairness. A large portion of this literature is focused on resource allocation systems, see e.g.,~\cite{bertsimas2011price,kleinberg1999fairness,NIPS2014_5588}. Group fairness in particular has been studied in the context of resource allocation problems~\cite{conitzer2019group,segal2018democratic,suksompong2018approximate}. 
A nascent stream of work proposes to impose fairness by means of constraints in an optimization problem, an approach we also follow. This is for example proposed in~\cite{aghaei2019learning}, and in~\cite{benabbou2018diversity,elzayn2019fair}, and in~\cite{ahmed2017diverse} for machine learning, resource allocation, and matching problems, respectively. Several authors have studied the price of fairness. In~\cite{bertsimas2011price}, the authors provide bounds for maximin fair optimization problems. Their approach is restricted to convex and compact utility sets. In~\cite{bei2019price}, the authors study price of fairness for indivisible goods with additive utility functions. In our graph covering problem, this property does not hold. Several authors have investigated notions of fairness under uncertainty, see e.g,~\cite{bateni2016fair,fish2019gaps,miyagishima2019fair,NIPS2014_5588}. These papers all assume full distributional information about the uncertain parameters and {cannot be employed} in our setting where limited data is available about node availability. Motivated by data scarcity, we take a robust optimization approach to model uncertainty which does not require distributional information. This problem is highly intractable due to the combinatorial nature of both the decision and uncertainty spaces. When fair solutions are hard to compute, ``approximately fair'' solutions have been considered~\cite{kleinberg1999fairness}. {In our work,} we adopt an approximation scheme. As such, {our approach} falls under the ``approximately fair'' category. {Recently, several} authors have emphasized the importance of fairness when conducting interventions in socially sensitive settings, see e.g.,~\cite{azizi2018designing,kube2019allocating,tsang2019group}. Our work most closely relates to~\cite{tsang2019group}, wherein the authors propose an algorithmic framework for fair influence maximization. We note that, in their work, nodes are not subject to failure and therefore their approach does not apply in our context.

\textit{Submodular Optimization.} One can view the group-fair maximum coverage problem as a multi-objective optimization problem, with the coverage of each community being a {separate objective}. In the deterministic case, this problem reduces to the multi-objective submodular optimization problem~\cite{chekuri2010dependent}, as coverage has the submodularity (diminishing returns) property. In addition, moderately sized problems of this kind can be solved optimally using integer programming technology. However, when considering uncertainty in node performance/availability, the objective function loses the submodularity property while exact techniques fail to scale to even moderate problem sizes. Thus, {existing (exact or approximate) approaches do not apply}. Our work more closely relates to the robust submodular optimization literature.  In~\cite{bogunovic2017robust,orlin2016robust}, the authors study the problem of choosing a set of up to~$I$ items, out of which~$J$ fail (which encompasses as a special case the robust graph covering problem \emph{without} fairness constraints). They propose a greedy algorithm with a constant (0.387) approximation factor, valid for $\textstyle J = o(\sqrt{I})$, and $\textstyle J = o(I)$, respectively. 
Finally, in~\cite{tzoumas2017resilient}, the authors propose another greedy algorithm with a general bound based on the curvature of the submodular function. These heuristics, although computationally efficient, are coverage-centered and do not take fairness into consideration. Thus, they may lead to discriminatory outcomes, see Table~\ref{table:DC-Greedy-Discrimination}.

\textit{Robust Optimization.}
Our solution approach closely relates to the robust optimization paradigm which is a computationally attractive framework for obtaining equivalent or conservative approximations based on duality theory, see e.g.,~\cite{ben2009robust,bertsimas2011theory,yanikouglu2019survey}. Indeed, we show that the robust graph covering problem can be written as a two-stage {robust problem} with binary second-stage decisions which is highly intractable in general~\cite{bertsimas2015design}. One stream of work proposes to restrict the functional form of the recourse decisions to functions of benign complexity~\cite{bertsimas2016multistage,bertsimas2018binary}. {Other} works rely on partitioning the uncertainty set into finite sets and applying constant decision rules on each partition~\cite{bertsimas2018binary,bertsimas2017data,hanasusanto2015k,postek2016multistage,vayanos2011decision}. The last stream of work investigates the so-called $K$-adaptability counterpart~\cite{bertsimas2010finite,chassein2018min,hanasusanto2015k,rahmattalabi2018robust,phebehanangelos}, in which $K$~candidate policies are chosen in the first stage and the best of these policies is selected \emph{after} the uncertain parameters are revealed. Our paper most closely relates {to~\cite{hanasusanto2015k,rahmattalabi2018robust}}. In~\cite{hanasusanto2015k}, the authors show that for bounded polyhedral uncertainty sets, {linear} two-stage robust optimization problem{s} can be approximately reformulated as MILPs. Paper~\cite{rahmattalabi2018robust} extends this result to a special case of discrete uncertainty sets. We prove that we can leverage this approximation to reformulate robust graph covering problem with fairness constraints \emph{exactly} for a much larger class of discrete uncertainty sets.



\section{Fair and Robust Graph Covering Problem}
\label{sec:problem_formulation}

We model a social network as a directed graph $\mathcal G = (\mathcal N, \mathcal E)$, in which $\mathcal N:=\{1,\ldots,N\}$ is the set of all nodes (individuals) and  $\mathcal E$ is the set of all edges (social ties). A directed edge from $\nu$ to $n$ exists, i.e., $(\nu, n) \in \mathcal E$, if node $n$ can be covered by $\nu$.  
We use $\delta(n):= \{\nu \in \mathcal N: (\nu, n)\in \mathcal{E}\}$ to denote the set of neighbors (friends) of $n$ in $\mathcal G$, i.e., the set of nodes that can cover node $n$. Each node $n \in \mathcal N$ is characterized by a set of attributes (protected characteristics) such as age, race, gender, etc., for which fair treatment is important. Based on these node characteristics, we partition $\mathcal N$ into $C$ disjoint groups $\mathcal N_c$, $c \in \mathcal C:=\{1,\ldots,C\}$, such that $\cup_{c\in \mathcal C} \mathcal N_c = \mathcal N$.

We consider the problem of selecting a set of $I$ {nodes} from $\mathcal N$ to act as `peer-monitors' for their neighbors, given that the availability of each node is unknown a-priori and at most $J$ nodes may fail (be unavailable). We encode the choice of monitors using a binary vector ${\bm x}$ of dimension $N$ whose $n$th element is one iff the $n$th node is chosen. We require ${\bm x} \in \mathcal X  := \{ {\bm x}\in \{0,1\}^N : {\textbf{e}^{\top}}{\bm x} \leq I \}$, where ${\textbf{{e}}}$ is a vector of all ones of appropriate dimension. Accordingly, we encode the (uncertain) node availability using a binary vector ${\bm \xi}$ of dimension $N$ whose $n$th element equals one iff node $n$ does not fail (is available). Given that data available to inform the distribution of ${\bm \xi}$ is typically scarce, we avoid making distributional assumptions on ${\bm \xi}$. Instead, we view uncertainty as deterministic and set based, in the spirit of robust optimization~\cite{ben2009robust}. Thus, we assume that ${\bm \xi}$ can take-on any value from the set $\Xi$ which is often referred to as the \emph{uncertainty set} in robust optimization. The set $\Xi$ may for example conveniently capture failure rate information. Thus, we require ${\bm \xi} \in \Xi := \{ {\bm \xi}\in \{0,1\}^N : {\textbf{e}}^\top {(\textbf{e} - \bm \xi)} \leq J\}$. A node $n$ is counted as `covered' if at least one of its neighbors is a monitor and does not fail (is available). We let ${\bm y}_n({\bm x},{\bm \xi})$ denote if $n$ is covered for the monitor choice ${\bm x}$ and node availability ${\bm \xi}$.
$$
{\bm y}_n({\bm x},{\bm \xi}):=  \textstyle \mathbb I \left( \sum_{\nu \in \delta (n)} {\bm \xi}_\nu {\bm x}_\nu \geq 1 \right).
$$
The coverage is then expressible as $F_\mathcal{G}( {\bm x}, {\bm \xi} ) := {\textbf{e}}^\top {\bm y}({\bm x},{\bm \xi})$. The \emph{robust covering problem} which aims to maximize the worst-case (minimum) coverage under node failures can be written as
\begin{equation}
\tag{$\mathcal{RC}$}
\max_{\bm x \in \mathcal X} \; \; \min_{\bm \xi \in \Xi} \; F_\mathcal{G}( {\bm x}, {\bm \xi} ).
\label{eq:robust_covering}
\end{equation}

Problem~\eqref{eq:robust_covering} ignores fairness and may result in discriminatory coverage with respect to (protected) node attributes {, see} Table~\ref{table:DC-Greedy-Discrimination}. We thus propose to augment the robust covering problem with fairness constraints. Specifically, we propose to achieve max-min fairness by imposing fairness constraints on each group's coverage: we require that {at least} a fraction~$W$ of {nodes from} each group be covered. In~\cite{tsang2019group}, the authors show that by conducting a binary search for the largest~$W$ for which fairness constraints are satisfied for all groups, the max-min fairness optimization problem is equivalent to the one with fairness constraints. Thus, we write the \emph{robust covering problem with fairness constraints} as
\begin{equation}
\tag{$\mathcal{RC}_{\text{fair}}$}
   \left\{ \max_{\bm x \in \mathcal X} \; \; \min_{\bm \xi \in \Xi} \; \sum_{c \in \mathcal C} F_{\mathcal{G},c}( {\bm x}, {\bm \xi} ) \; : \; F_{\mathcal{G},c}( {\bm x}, {\bm \xi} ) \geq W | \mathcal N_c | \quad \forall c \in \mathcal C, \; \forall {\bm \xi} \in \Xi \right\},
   \label{prob:single_stage_final}
\end{equation}
where $\textstyle F_{\mathcal{G},c}({\bm x}, {\bm \xi}) := \sum_{n \in \mathcal N_c} {\bm y}_n({\bm x}, {\bm \xi})$ is the coverage of group $c \in \mathcal C$. Note that if $|\mathcal C|=1$, Problem~\eqref{prob:single_stage_final} reduces to Problem~\eqref{eq:robust_covering}, and if $\Xi=\{{\textbf{e}} \}$, Problem~\eqref{prob:single_stage_final} reduces to the deterministic covering problem with fairness constraints. We emphasize that our approach can handle fairness with respect to more than one protected attribute by either: \emph{(a)} partitioning the network based on joint values of the protected {attributes and} imposing a max-min fairness constraint for each group; or \emph{(b)} imposing max-min fairness constraints for each protected attribute separately. Problem~\eqref{prob:single_stage_final} is computationally hard due to the combinatorial nature of both the uncertainty and decision spaces. Lemma~\ref{lem:np-hard} characterizes its complexity. Proofs of all results are in the supplementary document.
\begin{lemma}
Problem~\eqref{prob:single_stage_final} is $\mathcal{NP}$-hard.
\label{lem:np-hard}
\end{lemma}


\section{Price of Group Fairness}\label{sec:PoF}

In Section~\ref{sec:problem_formulation}, we proposed a novel formulation of the robust covering problem incorporating fairness constraints, Problem~\eqref{prob:single_stage_final}. Unfortunately, adding fairness constraints to Problem~\eqref{eq:robust_covering} comes at a price to overall worst-case coverage. In this section, we study this {\emph{price of group fairness}.}

\begin{definition}
Given a graph $\mathcal G$, the Price of Group Fairness ${\text{\rm{PoF}}}(\mathcal G, I, J)$ {is} the ratio of the coverage loss due to fairness constraints to the maximum coverage in the absence of fairness constraints, i.e.,
\begin{equation}
{\text{\rm{PoF}}}(\mathcal G, I, J) := 1 - \frac{\text{\rm{OPT}}^{\text{\rm{fair}}}(\mathcal G, I, J)}{\text{\rm{OPT}}(\mathcal G, I, J)},
\label{equ:POF_deterministic}
\end{equation}
where ${\text{\rm{OPT}}}^{\text{\rm{fair}}}(\mathcal G, I, J)$ and ${\text{\rm{OPT}}}(\mathcal G, I, J)$ denote the optimal objective values of Problems~\eqref{prob:single_stage_final} and \eqref{eq:robust_covering}, respectively, when $I$ monitors can be chosen and at most $J$ of them may fail.
\end{definition}

In this work, we are motivated by applications related to social networks, where it has been observed that people with similar (protected) characteristics tend to interact more frequently with one another, forming friendship groups (communities). This phenomenon, known as \emph{homophily}~\cite{mcpherson2001birds}, has been observed for characteristics such as race, gender, education, etc.\cite{currarini2009economic}.
This motivates us to study the PoF in Stochastic Block Model (SBM) networks~\cite{fienberg1981categorical}, a widely accepted model for networks with community structure. In SBM networks, nodes are partitioned into $C$ disjoint communities $\mathcal N_c$, $c \in \mathcal C$. Within each community $c$, an edge between two nodes is present independently with probability $p_c^{\text{in}}$. Between a pair of communities $c$ and  $c' \in \mathcal C$, edges exist independently with probability $p_{cc'}^{\text{out}}$ and we typically have $p_c^{\text{in}} > p_{cc'}^{\text{out}}$ to capture homophily. Thus, SBM networks are very adequate models for our purpose. We assume w.l.o.g.\ that the communities are labeled such that: $|\mathcal N_1| \leq \ldots \leq |\mathcal N_C|$. 


\textbf{Deterministic Case.}
We first study the PoF in the deterministic case for which $J=0$. Lemma 2 shows that there are worst-case networks for which PoF can be arbitrarily bad.
\begin{lemma}
Given $\epsilon>0$, there exists a budget $I$ and a network $\mathcal G$ with $\textstyle N \geq \frac{4}{\epsilon} + 3$ nodes such that ${\rm{PoF}}(\mathcal G, I,0) \geq 1-\epsilon.$
\label{lem:worst-case}
\end{lemma}
Fortunately, as we will see, this pessimistic result is not representative of the networks that are seen in practice. We thus investigate the loss in expected coverage due to fairness constraints, given by
\begin{equation}
\overline{ \text{\rm{PoF}}}(I, J) := 1 - \frac{\EX_{\mathcal G\sim \text{SBM}}[\text{OPT}^{\text{fair}}(\mathcal G, I, J)]}{\EX_{\mathcal G\sim \text{SBM} }[\text{OPT}(\mathcal G, I, J)]}.
\label{equ:POF_average}
\end{equation}
We emphasize that we investigate the loss in the expected coverage rather than the expected {PoF} for analytical tractability reasons. We make the following assumptions about SBM network.
\begin{assumption}
    For all communities $c\in \mathcal C$, the probability of an edge between two individuals in the community is inversely proportional to the size of the community, 
    i.e., $p_{c}^{\text{\rm{in}}} = \Theta(|\mathcal{N}_c|^{-1})$.
    \label{ass:pin}
\end{assumption} 
\begin{assumption}
    For any two communities $c, c' \in \mathcal C$, the probability of an edge between two nodes $n \in \mathcal N_c$ and $\nu \in \mathcal N_{c'}$ is $p_{cc'}^{\text{\rm out}} = O( (|\mathcal N_c| \log^2 |\mathcal N_c|)^{-1})$.
    \label{ass:pout}
\end{assumption}
Assumption 1 is based on the observation that social networks are usually sparse. This means that most individuals do not form too many links, even if the size of the network is very large. Sparsity is characterized in the literature by the number of edges being proportional to the number of nodes which is the direct result of Assumption~\ref{ass:pin}. Assumption~\ref{ass:pout} is necessary for meaningful community structure in the network. We now present results for the upper bound on {$\overline{\text{PoF}}$} in SBM networks.
\begin{proposition}
Consider an SBM network model with parameters $p_{c}^{\text{\rm{in}}}$ and $p_{cc'}^{\text{\rm{out}}}$, $c,c' \in \mathcal{C}$, satisfying Assumptions~\ref{ass:pin} and~\ref{ass:pout}. If $I = O(\log N)$, then
$$
\overline{\rm{PoF}}(I, 0)  =  1 - \frac{\sum_{c\in\mathcal C}{|\mathcal N_c|}}{{\sum_{c\in\mathcal C}{|\mathcal N_c|d(C)/d(c)}}} - o(1), \text{ where  }\textstyle d(c) := \log |\mathcal N_c|({\log\log |\mathcal N_c|})^{-1}.
$$
\label{prop:PoF}
\end{proposition}
\textit{Proof Sketch.} First, we show that under Assumption~\ref{ass:pin}, the coverage within each community is the sum of the degrees of the monitoring nodes. Then, using the assumption on~$I$ in the premise of the proposition (which can be interpreted as a ``small budget assumption''), we evaluate the maximum coverage within each community. Next, we show that between-community coverage is negligible compared to within-community coverage. Thus, we determine the distribution of the monitors, in the presence and absence of fairness constraints. PoF is computed based on the these two quantities. $\blacksquare$


\textbf{Uncertain Case.} Here, imposing fairness is more challenging as we do not know a-priori which nodes may fail. Thus, we must ensure that fairness constraints are satisfied under all failure scenarios.
\begin{proposition}
Consider an SBM network model with parameters $p_{c}^{\text{\rm{in}}}$ and $p_{cc'}^{\text{\rm{out}}}$, $c,c' \in \mathcal{C}$, satisfying Assumptions~\ref{ass:pin} and~\ref{ass:pout}. If $I=O(\log N)$, then
$$
\overline{\rm{PoF}}(I, J) = 1 - \frac{\eta\sum_{c\in\mathcal C}{|\mathcal N_c|}}{(I-J)\times{{d(C)}}} - \frac{J \sum_{c\in\mathcal C\setminus {\{C\}}}{d(c)}}{(I-J) \times d(C)} - o(1),
$$
where $d(c)$ is as in Proposition~\ref{prop:PoF} and $\textstyle \eta := {(I-CJ)}\left({\sum_{c\in\mathcal C}{|\mathcal N_c|/d(c)}}\right)^{-1}$.
\label{prop:price_of_fairness_uncertain}
\end{proposition}
%
\textit{Proof Sketch.} The steps of the proof are similar to those in the proof of Proposition~\ref{prop:PoF} with the difference that, under uncertainty, monitors should be distributed such that the fairness constraints are satisfied even after~$J$ nodes fail. Thus, we quantify a minimum number of monitors that should be allocated to each community. We then determine the worst-case coverage both in the presence and absence of fairness constraints. PoF is computed based on these two quantities. $\blacksquare$

\begin{wrapfigure}{r}{0.5\textwidth}
  \begin{center}
  \includegraphics[width=0.45\textwidth]{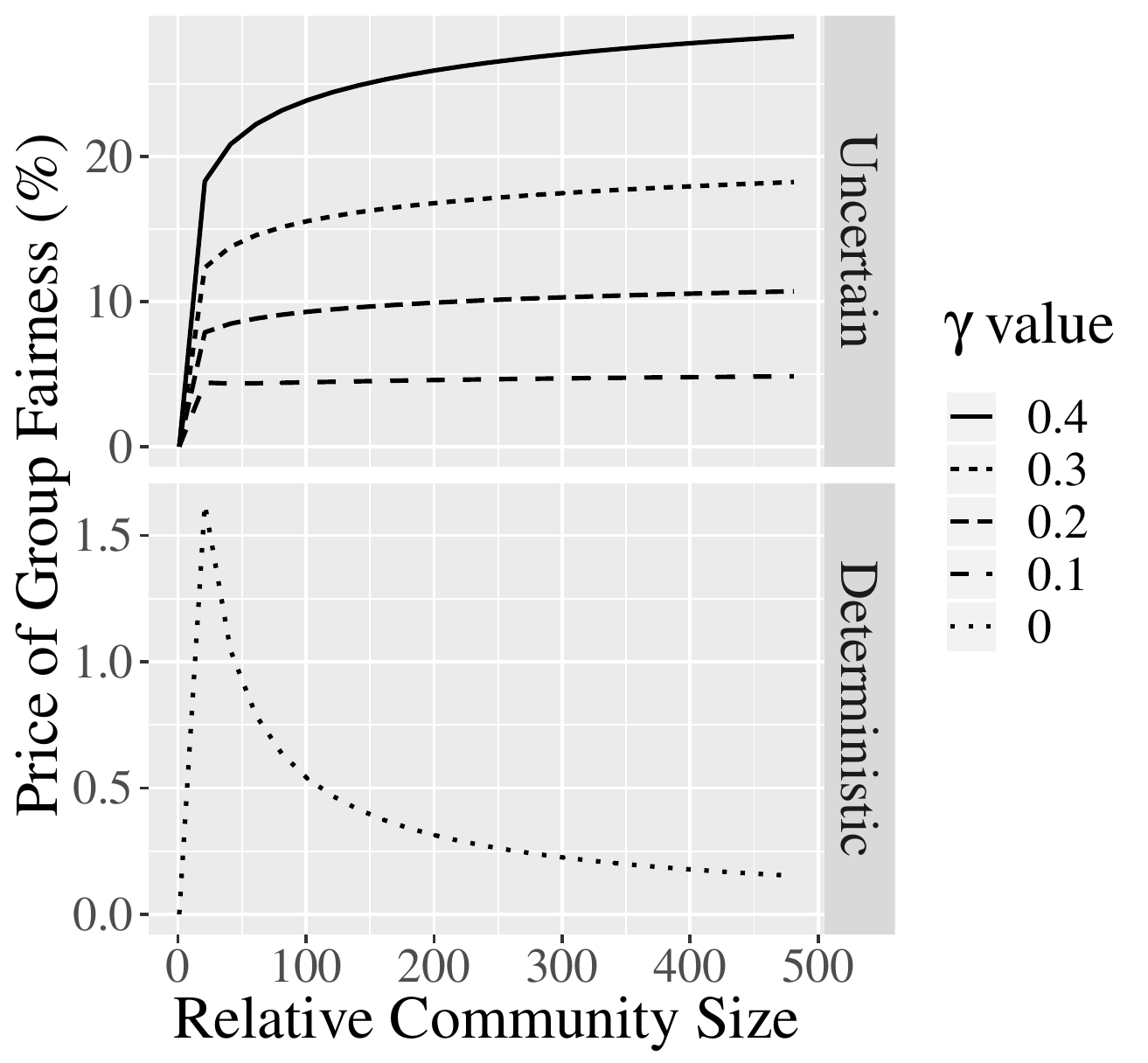}
  \end{center}
  \caption{PoF in the {uncertain {(top)} and deterministic {(bottom)}} settings for SBM networks consisting of two communities ($\mathcal C=\{1,2\}$) where the size of the first community is fixed at $|\mathcal N_1| = 20$ and the size of the other community is increased from $|\mathcal N_2 | = 20$ to $10,000$. In the uncertain setting, $\gamma$ denotes the fraction of nodes that fail.}
  \label{fig:price}
\end{wrapfigure}


Propositions 1 and 2 show how PoF changes with the relative sizes of the communities for the deterministic and uncertain cases, respectively. Our analysis shows that without fairness, one should place all the monitors in the biggest community. {Under a} fair allocation however {monitors are more evenly distributed (although larger communities still receive a bigger share).} {Figure~\ref{fig:price} illustrates} the PoF results in the case of two communities for different failure rates $\gamma$ ($J = \gamma I$), ignoring the $o(.)$ order terms. We keep the size of the first (smaller) community fixed and vary the size of the larger community. In both cases, if $|\mathcal N_1| = |\mathcal N_2|$, the PoF is zero since uniform distribution of monitors is optimal. 
As $|\mathcal N_2|$ increases, the PoF increases in both cases. Further increases in $|\mathcal N_2|$ {result} in a decrease in the PoF for the deterministic case{: under a fair allocation, the bigger community receives a higher share of monitors which is aligned with the total coverage objective.} Under uncertainty however, the PoF is non-decreasing: to guarantee fairness, additional monitors must be allocated to the smaller groups. This also explains why PoF increases with $\gamma$.

\section{Solution Approach}\label{sec:solution-approach}
Given the intractability of Problem~\eqref{prob:single_stage_final}, see Lemma~\ref{lem:np-hard}, we adopt a conservative approximation approach. To this end, we proceed in three steps. First, we note that a difficulty of Problem~\eqref{prob:single_stage_final} is the discontinuity of its objective function. Thus, we show that~\eqref{prob:single_stage_final} can be formulated equivalently as a \emph{two-stage} robust optimization problem by introducing a \emph{fictitious} counting phase \emph{after} ${\bm \xi}$ is revealed. Second, we propose to approximate this decision made in the counting phase (which decides, for each node, whether it is or not covered). Finally, we demonstrate that the resulting approximate problem can be formulated equivalently as a moderately sized MILP, {wherein the trade-off between suboptimality and tractability} can be controlled by a single design parameter.

\textbf{Equivalent Reformulation.}
For any given choice of ${\bm x}\in \mathcal X$ and ${\bm \xi} \in \Xi$, the objective $F_{\mathcal G}({\bm x},{\bm \xi})$ can be explicitly expressed as the optimal objective value of a covering problem. As a result, we can express~\eqref{prob:single_stage_final} equivalently as the two-stage \emph{linear} robust problem
\begin{equation}
\begin{array}{clll}
\displaystyle \max_{\bm x \in \mathcal{X}} \; \min_{\bm \xi \in \Xi} \; \max_{\bm y \in \mathcal Y} \left\{\sum_{n \in \mathcal N} {\bm y}_n \; :  {\bm y}_n \leq \sum_{\nu \in \delta(n)} {\bm \xi}_{\nu} {\bm x}_{\nu}, \; \forall n \in \mathcal N \right \},
\end{array}
\label{prob:two_stage_final}
\end{equation}
{see Proposition~\ref{prop:equivalent-formulation} below.} The second-stage binary decision variables $\textstyle {\bm y}  \in \mathcal Y:=\{ {\bm y} \in \{0,1\}^{N}: \sum_{n\in\mathcal{N}_{c}} {\bm y}_{n} \geq W|\mathcal {N}_{c}|, \; \forall c \in \mathcal C \}$ admit a very natural interpretation: at an optimal solution, ${\bm y}_n=1$ if and only if node $n$ is covered. Henceforth, we refer to ${\bm y}$ as a \emph{covering scheme}.

\begin{definition}[Upward Closed Set]
A set $\mathcal X$ given as a subset of the partially ordered set $[0,1]^{N}$ equipped with the element-wise inequality, is said to be upward closed if for all $\bm x \in \mathcal X$ and $\bar{\bm x} \in [0,1]^N$ such that $\bar{\bm x} \geq \bm x$, it holds that $\bar{\bm x}\in \mathcal X$.
\end{definition}
Intuitively, sets involving lower bound constraints on the (sums of) parameters satisfy this definition. For example, sets that require a minimum fraction of
nodes to be available. We can also consider group-based availability and require a minimum fraction of nodes to be available in every group.  
\begin{assumption}
We assume that: 
The set $\Xi$ is defined through $\Xi := \{0,1\}^N \cap \mathcal T$ for some upward closed set~$\mathcal T$ given by~$\mathcal T:=\{{\bm \xi} \in \mathbb R^N  :  {\bm A} {\bm \xi} \geq {\bm b}\}$, with ${\bm A} \in \mathbb R^{R \times N}$ and ${\bm b} \in \mathbb R^R$. \label{assumption:XiandY}
\end{assumption}
\begin{proposition}
Problems~\eqref{prob:single_stage_final} and \eqref{prob:two_stage_final} are equivalent.
\label{prop:equivalent-formulation}
\end{proposition}
\textbf{$K$-adaptability Counterpart.} {Problem~\eqref{prob:two_stage_final} has the advantage of being linear. Yet, its max-min-max structure precludes us from solving it directly. We investigate a conservative approximation to Problem~\eqref{prob:two_stage_final} referred to as \emph{$K$-adaptability counterpart}, wherein $K$ \emph{candidate} covering schemes are chosen in the first stage and the best (feasible and most accurate) of those candidates is selected after ${\bm \xi}$ is revealed. Formally, the $K$-adaptability counterpart of Problem~\eqref{prob:two_stage_final} is
%
%
\begin{equation}
\begin{array}{clll}
\displaystyle \mathop  \text{maximize}_{ \begin{smallmatrix} {\bm x} \in \mathcal X \\ {\bm y}^k \in \mathcal Y, \; k\in \mathcal K
\end{smallmatrix}  } & \displaystyle \min_{\bm \xi \in \Xi} \; \displaystyle \max_{k \in \mathcal K} \; \left\{\sum_{n \in \mathcal N} {\bm y}^{k}_n \; : \; {\bm y}^{k}_n \leq \sum_{\nu \in \delta(n)} {\bm \xi}_{\nu} {\bm x}_{\nu} \; \; \forall n \in \mathcal N\right\}, 
\end{array}
\label{prob:two_stage_k_adapt}
\end{equation}
where ${\bm y}^k$ denotes the $k$th candidate covering scheme, $k \in \mathcal K$. 
We emphasize that the covering schemes are not inputs but rather \emph{decision variables} of the $K$-adaptability problem. Only the value~$K$ is an input. The optimization problem will identify the best $K$ covering schemes that satisfy all the constraints including fairness constraints. The trade-off between optimality and computational complexity of Problem~\eqref{prob:two_stage_k_adapt} can conveniently be tuned using the single parameter~$K$. } 

\textbf{Reformulation as an MILP.} We derive an exact reformulation for the $K$-adaptability counterpart~\eqref{prob:two_stage_k_adapt} of the \emph{robust covering problem} as a moderately sized MILP. Our method extends the results from \cite{rahmattalabi2018robust} to significantly more general uncertainty sets that are useful in practice, and to problems involving constraints on the set of covered nodes. Henceforth, we let $\mathcal{L} := \{0,\ldots,N\}^{K}$, and we define $\mathcal{L}_{+}:=\{\lb\in\mathcal L: \lb > \bm 0\} \text{ and } \mathcal{L}_{0} := \{\lb\in\mathcal L: \lb \ngtr \bm 0\}.$ We present a variant of the generic $K$-adaptability Problem~\eqref{prob:two_stage_k_adapt}, where the uncertainty set $\Xi$ is parameterized by vectors $\lb \in \mathcal L.$  Each $\lb $ is a $K$-dimensional vector, whose $k$th component encodes if the $k$th covering scheme satisfies the constraints of the second stage maximization problem. In this case, $\lb_{k} = 0$. Else, if the $k$th covering scheme is infeasible, $\lb_{k}$ is equal to the index of a constraint that is violated.

\begin{theorem}
Under Assumption~\ref{assumption:XiandY},  Problem~\eqref{prob:two_stage_k_adapt} is equivalent to the mixed-integer bilinear program 
\begin{equation}
\renewcommand{\arraystretch}{1.5}
\begin{array}{clll}
\max &\tau &\\
\text{\rm{s.t.}} &  \tau \in \mathbb R, \; {\bm x} \in \mathcal X, \; {\bm y}^k \in \mathcal Y \;\; \forall k \in \mathcal K & \\
& \!\!\! \left. \begin{array}{l}
{\bm \theta}({\bm \ell}), \; {\bm \beta}^k({\bm \ell}) \in \mathbb R^N_+, \;  {\bm \alpha}(\bm \ell) \in \mathbb R^R_+, \; {\bm \nu}({\bm \ell}) \in \mathbb R^K_+, \; {\bm \lambda}(\lb) \in \Delta_K(\bm \ell) \\
\displaystyle \tau \; \leq \;  - \displaystyle \emph{\textbf{e}}^{\top} {\bm\theta}(\bm \ell) + {\bm \alpha}({\bm \ell})^{\top} {\bm b} - \sum_{
\begin{smallmatrix} k\in\mathcal K : \\ {\bm \ell}_{k} \neq 0 \end{smallmatrix}} \left( {\bm y}^{k}_{{\bm \ell}_k} -1 \right) {\bm \nu}_k(\bm \ell) + \ldots  \\
\qquad  \qquad \qquad \ldots + \sum_{ \begin{smallmatrix} k \in \mathcal K : \\ {\bm \ell}_k = 0 \end{smallmatrix} } \sum_{n\in\mathcal N}{\bm y}^k_n {\bm \beta}_n^k({\bm \ell}) + \sum_{k\in\mathcal K} {\bm \lambda}_k({\bm \ell}) \sum_{n\in \mathcal N} {\bm y}^k_n \\
{\bm \theta}_n (\bm \ell)  \; \leq \; {\bm A}^\top {\bm \alpha}(\bm \ell) + \displaystyle\sum_{\begin{smallmatrix} k\in\mathcal {K} : \\ {\bm \ell}_k \neq 0 \end{smallmatrix}} \sum_{ \nu \in\delta({\bm \ell}_k)} {\bm x}_\nu {\bm \nu}_k ({\bm \ell}) - \displaystyle \sum_{\begin{smallmatrix} k\in\mathcal K : \\ {\bm \ell}_k = 0
\end{smallmatrix}} \sum_{\nu\in\delta(n)}{\bm x}_{\nu} {\bm \beta}_n^k({\bm \ell}) \;\; \forall n \in \mathcal N 
\end{array}  \right\} \forall \lb \in  \mathcal{L}_{0}  \\
& \!\!\! \left. \begin{array}{l}
{\bm \theta}({\lb})\in \mathbb R_+^N,\; {\bm \alpha}(\bm \ell) \in \mathbb R_+^R, \; {\bm \nu}(\bm \ell) \in \mathbb R_+^K \\
1 \leq - \textbf{\emph{e}}^\top {\bm\theta} (\bm \ell) + {\bm \alpha}(\bm \ell)^\top {\bm b}  - \sum_{\begin{smallmatrix} k\in\mathcal {K} : \\ {\bm \ell}_k \neq 0 \end{smallmatrix}} \left( {\bm y}^k_{{\bm \ell}_k} - 1\right) {\bm \nu}_k({\bm \ell}) \\
{\bm \theta}_n({\bm \ell}) \; \leq \; {\bm A}^\top {\bm \alpha}({\bm \ell}) + \displaystyle \sum_{\begin{smallmatrix}
k\in\mathcal K : \\ {\bm \ell}_k \neq 0 \end{smallmatrix}}
\sum_{ \nu \in \delta({\bm \ell}_k)} {\bm x}_\nu {\bm \nu}_k({\bm \ell})  \;\; \forall n \in \mathcal N  \qquad  
\end{array} \right\} \forall \lb \in \mathcal L_+,
\end{array}
\label{prob:K-adaptability_MILP}
\end{equation}
which can be reformulated equivalently as an MILP using standard ``Big-$M$'' techniques  
since all bilinear terms are products continuous and binary variables. The size of this MILP scales with $\textstyle |\mathcal L| = (N + 1)^{K}$; it is polynomial in all problem inputs for any fixed $K$.
\label{thm:main}

\end{theorem}

\begin{proof}[Proof Sketch]
The reformulation relies on three key steps: First, we partition the uncertainty set {by using the parameter ${\bm \ell}$.} 
Next, we show that by relaxing the integrality constraint on the uncertain parameters $\bm \xi,$ the problem remains unchanged. {This is the key result that enables us to provide an equivalent formulation for Problem~\eqref{prob:two_stage_k_adapt}.} Finally, we employ linear programming duality theory, to reformulate the robust optimization formulation over each subset. As a result, the formulation {has} two sets of decision variable: \emph{(a)} The decision variables of the original problem; \emph{(b)} Dual variables {parameterized} by~${\bm \ell}$ which emerge from the dualization. 
\end{proof}

\textbf{Bender's Decomposition.} {In Problem~\eqref{prob:K-adaptability_MILP},} once binary variables ${\bm x}$ and $\{{\bm y}^k\}_{k \in \mathcal K}$ are fixed, the problem decomposes across $\bm \ell$, i.e., all remaining variables are real valued and can be found by solving a linear program for each $\bm \ell$. Bender's decomposition is an \emph{exact} solution technique that leverages such decomposable structure for more efficient solution~\cite{benders1962partitioning,bertsimas1997introduction}. 
Each iteration of the algorithm starts with {the} solution of {a} relaxed master problem, which is fed into the subproblems to identify violated constraints to add to the master problem. The process repeats until no more violated constraints can be identified. The formulations of master and subproblems are provided in~Section~\ref{sec:Benders-problem-master-sub}.

\textbf{Symmetry Breaking Constraints. } Problem \eqref{prob:K-adaptability_MILP} presents a large amount of symmetry. Indeed, given~$K$ candidate covering schemes $\bm y^{1}, \ldots, \bm y^{K}$, their indices can be permuted to yield another, distinct, feasible solution with identical cost. The symmetry results in significant slow down of the Brand-and-Bound procedure~\cite{bertsimas2005optimization}. Thus, we introduce symmetry breaking constraints in the formulation~\eqref{prob:K-adaptability_MILP} that stipulate the candidate covering schemes be lexicographically decreasing. We refer to~\cite{phebehanangelos} for details.



\section{Computational Study on Social Networks of Homeless Youth}
\label{sec:computational_study}

We evaluate our approach on the five social networks from Table~\ref{table:DC-Greedy-Discrimination}. Details on the data are provided in Section~\ref{sec:final-experiment}. We investigate the robust graph covering problem with maximin racial fairness constraints. All experiments were ran on a Linux 16GB RAM machine with Gurobi~v6.5.0.


First, we compare the performance of our approach against the greedy algorithm of~\cite{tzoumas2017resilient} and the degree centrality heuristic (DC). The results are summarized in Figure~\ref{fig:PoF} (left). From the figure, we observe that an increase in $K$ results in an increase in performance along both axes, with a significant jump from $K=1$ to $K=2, 3$ (recall that $K$ controls complexity/optimality trade-off of our approximation). We note that the gain starts diminishing from $K=2$ to $K=3$. Thus, we only run up to $K=3$. In addition the computational complexity of the problem increases exponentially with $K$, limiting us to increase $K$ beyond 3 for the considered instances. As demonstrated by our results, $K \sim 3$ was sufficient to considerably improve fairness of the covering at moderate price to efficiency. Compared to the baselines, with $K=3$, we significantly improve the coverage of the worse-off group over greedy (resp.\ DC) by 11\% (resp.\ 23\%) on average across the five instances. 

\begin{figure}
\includegraphics[width=\textwidth]{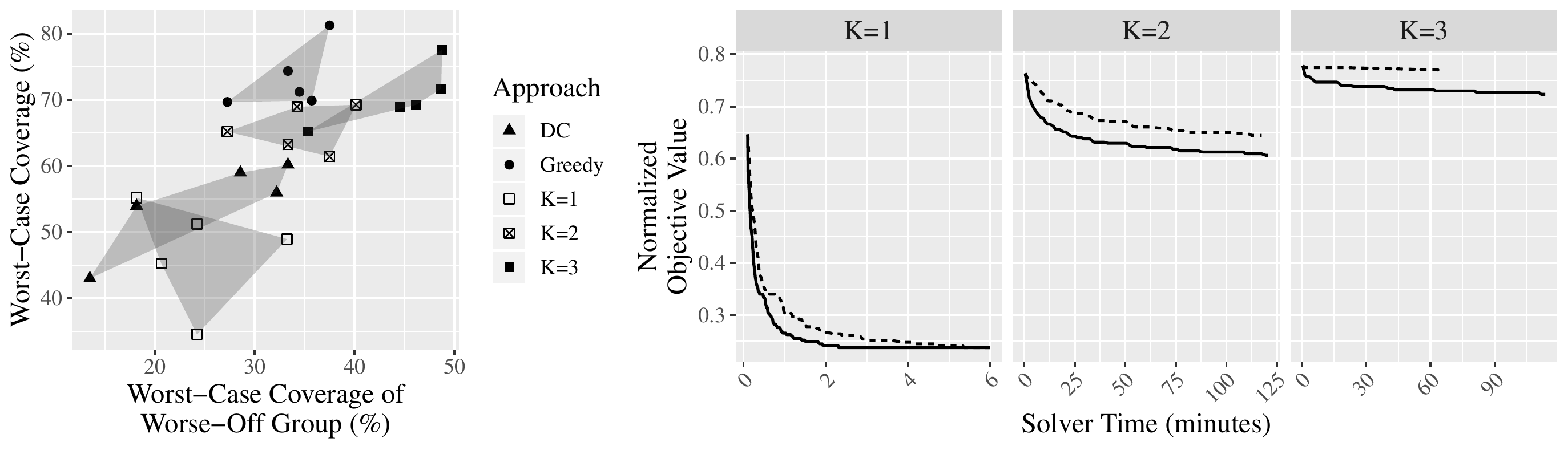}
\caption{Left figure: Solution quality (overall worst-case coverage versus worst-case coverage of the group that is worse-off) for each approach (DC, Greedy, and $K$-adaptability for $K=1,2,3$); The points represent the results of each approach applied to each of the five real-world social networks from Table~\ref{table:DC-Greedy-Discrimination}; Each shaded area corresponds to the convex hull of the results associated with each approach; Approaches that are more fair (resp.\ efficient) are situated in the right- (resp.\ top-)most part of the graph. Right figure: Average of the ratio of the objective value of the master problem to the network size (across the five instances) in dependence of solver time for the Bender's decomposition approach (dotted line) and the Bender's decomposition approach augmented with symmetry breaking constraints (solid line). For both sets of experiments, the setting was $I = N/3$ and $J = 3$.}
\label{fig:PoF}
\end{figure}

\begin{table}
  \small
  \centering
  \begin{tabular}{r*{13}{c}}
    \toprule
    \multirow{4}{*}{Name} & \multirow{4}{*}{Size $N$} &  \multicolumn{6}{c}{ Improvement in Min. Percentage Covered (\%)} & \multicolumn{6}{c}{PoF (\%)}
    \\
    \cmidrule(lr){3-14}
    &  &  \multicolumn{6}{c}{Uncertainty Level $J$} & \multicolumn{6}{c}{Uncertainty Level $J$}\\
    \cmidrule(lr){3-8}\cmidrule(lr){9-14}
    & & 0 & 1 & 2 & 3 & 4 & 5 & 0 & 1 & 2 & 3 & 4 & 5  \\
    \hline
    \texttt{SPY1} & 95 & 15 & 16 & 14 & 10 & 10 & 9 & 1.4 & 1.0 & 2.1 & 1.3 & 3.3 & 4.2  \\
    \texttt{SPY2} & 117 & 20 & 14 & 9 & 10 & 8 & 10 & 0.0 & 1.2 & 3.7 & 3.3 & 3.6 & 3.7  \\
    \texttt{SPY3} & 118 & 20 & 16 & 16 & 15 & 11 & 10 & 0.0 & 3.4 & 4.8 & 6.4 & 3.2 & 4.0\\
    \texttt{MFP1} & 165 & 17 & 15 & 7 & 11 & 14 & 9 & 0.0 & 3.1 & 5.4 & 2.4 & 6.3 & 4.4    \\
    \texttt{MFP2} & 182 & 11 & 12 & 10 & 9 & 12 & 12 & 0.0 & 1.0 & 1.0 & 2.2 & 2.4 & 3.6\\
    \hline
    \multicolumn{2}{c}{{Avg.} ($I = N/3$)}  & 16.6  & 14.6 & 11.2 & 11.0 & 11.0 & 10.0 & 0.3 & 1.9 & 3.4 & 3.1 & 3.8 & 4.0 \\
    \hline 
    \hline \multicolumn{2}{c}{{Avg. ($I = N/5)$}} & {15.0} & {13.8} & {14.0} & {10.0} & {9.0} & {6.7} & {0.6} & {2.1} & {3.2} & {3.2} & {3.9} & {3.8} \\
    \hline
    \hline \multicolumn{2}{c}{{Avg. ($I = N/7)$}} & {12.2} & {11.4} & {11.2} & {11.4} &  {8.2} & {6.4} & {0.1} & {2.5} & {3.5} & {3.2} & {3.5} & {4.0} \\
    \bottomrule
  \end{tabular}
  \vspace{0.1cm}
  \caption{Improvement on the worst-case coverage of the worse-off group and associated PoF for each of the five real-world social networks from Table~\ref{table:DC-Greedy-Discrimination}. The first five rows correspond to the setting $I=N/3$. In the interest of space, we only show averages for the settings $I=N/5$ and $I=N/7$. In the deterministic case ($J=0$), the PoF is measured relative the coverage of the true optimal solution (obtained by solving the integer programming formulation of the graph covering problem). In the uncertain case ($J>0$), the PoF is measured relative to the coverage of the greedy heuristic of~\cite{tzoumas2017resilient}.}
\label{table:K-adaptability-Discrimination}
\end{table}


Second, we investigate the effect of uncertainty on the coverage of the worse-off group and on the PoF, for both the deterministic ($J=0$) and uncertain ($J>0$) cases as the number of monitors~$I$ is varied in the set $\{N/3, N/5,N/7\}$. These settings are motivated by numbers seen in practice (typically, the number of people that can be invited is 15-20\% of network size). Our results are summarized in Table~\ref{table:K-adaptability-Discrimination}. Indeed, from the table, we see for example that for $I=N/3$ and $J=0$ our approach is able to improve the coverage of the worse-off group by 11-20\% and for $J>0$ the improvement in the worse-case coverage of the worse-off group is 7-16\%. On the other hand, the PoF is very small: 0.3\% on average for the deterministic case and at most 6.4\% for the uncertain case. These results are consistent across the range of parameters studied.  We note that the PoF numbers also match our analytical results on PoF in that uncertainty generally induces higher PoF.


Third, we perform a head-to-head comparison of our approach for $K=3$ with the results in Table~\ref{table:DC-Greedy-Discrimination}. Our findings are summarized in Table~\ref{table:HEA-TO-HEAD} in {Section~\ref{sec:final-experiment}}. As an illustration, in \texttt{SPY3}, the worst-case coverage by racial group under our approach is: White 90\%, Hispanic 44\%, Mixed 85\% and Other 87\%. These numbers suggest that coverage of Hispanics (the worse-off group) has increased from 33\% to 44\%, a significant improvement in fairness. To quantify the overall loss due to fairness, we also compute PoF values. The maximum PoF across all instances was at most 4.2\%, see Table~\ref{table:HEA-TO-HEAD}. 


Finally, we investigate the benefits of augmenting our formulation with symmetry breaking constraints. Thus, we solve all five instances of our problem with the Bender's decomposition approach with and without symmetry breaking constraints. The results are summarized in Figure~\ref{fig:PoF} (right). Across our experiments, we set a time limit of 2 hours since little improvement was seen beyond that. In all cases, and in particular for $K=2$ and~$3$, symmetry breaking results in significant speed-ups. For $K=3$ (and contrary to Bender's decomposition augmented with symmetry breaking), Bender's decomposition alone fails to solve the master problem to optimality within the time limit. We would like to remark that employing $K$-adaptability is necessary: indeed, Problem~\eqref{prob:single_stage_final} would not fit in memory. Similarly, using Bender's decomposition is needed: even for moderate values of $K$ (2 to 3), the $K$-adaptability MILP~\eqref{prob:K-adaptability_MILP} could not be loaded in memory.


{\textbf{Conclusion.} We believe that the robust graph covering problem with fairness constraints is worthwhile to investigate. It poses a huge number of challenges and holds great promise in terms of the realm of possible real-world applications with important potential societal benefits, e.g., to prevent suicidal ideation and death and to protect individuals during disasters such as landslides.}


\newpage

\section*{Acknowledgements}
We are grateful to three anonymous referees whose comments helped substantially improve the quality of this paper. This work was supported by the Smart \& Connected Communities program of the National Science Foundation under NSF award No. 1831770 and by the US Army Research Office under grant number W911NF1710445.

\bibliographystyle{plain}
\bibliography{nips2019.bib}

\newpage

\appendix



\section{Supplemental Material: Experimental Results in Section~\ref{sec:computational_study}}
\label{sec:final-experiment}

\textbf{Data and Data Preprocessing.} The original datasets used throughout our paper are described in detail in~\cite{barman2016sociometric}. They present 8 racial groups, with each individual belonging to a single group. To avoid misinterpretation of the results, we collect racial groups with a population $<10\%$ of the network size~$N$ under the ``Other'' category. The racial composition of the networks after the preprocessing is provided in Table~\ref{table:racial_composition}. For instance, network \texttt{SPY1} consists of~$54\%$ White, $11\%$ Black, $15\%$ Mixed and $20\%$ Others. The empty entry for Hispanic indicates that their population was less than $10\%$; as a result, they are categorized under ``Other''.

\begin{table}[ht] 
  \centering
  \small
  \begin{tabular}{cccccc}
    \toprule
    {Network Name} & White & Black & Hispanic & Mixed & Other \\
    \hline
    \texttt{SPY1} & 54 & 11  & -- & 15 & 20  \\
    \texttt{SPY2} & 55 & -- & 11 &  21 & 13  \\
    \texttt{SPY3} & 58 & -- & 10 & 18 & 14 \\
    \texttt{MFP1} & 16 & 38 & 22 & 16 & 8 \\
    \texttt{MFP2} & 16 & 32  & 22 & 20 & 10 \\  
    \toprule
    \end{tabular}
  \caption{Racial composition (\%) of the social networks considered after preprocessing}
  \label{table:racial_composition}
  \vspace*{-\baselineskip} 
\end{table}


\textbf{Setting of Parameter $W$.} We now describe in detail the procedure we use to select~$W$ in our experiments. As noted in Section~\ref{sec:problem_formulation}, to achieve maximin fairness, $W$ must take the maximum value for which the problem is feasible (fairness constraints satisfied). Its value thus depends on other parameters, including~$I$, $J$, and $K$. In our experiments, we conduct a search to identify the best value of~$W$ for each setting. Specifically, we vary~$W$ from 0 to 1, in increments of 0.04; we employ the largest~$W$ for which the problem is feasible. By construction, this choice of~$W$ guarantees that all of the fairness constraints are satisfied. In Table~\ref{table:W}, we provide the values of~$W$ associated with the results in Table~\ref{table:K-adaptability-Discrimination} for $I = N/3$ and $K=3$ and for each of the values of~$J$. 
\begin{table}[ht] 
  \centering
  \small
  \begin{tabular}{cccccc}
    \toprule
    {Network Name} & $J=1$ & $J=2$ & $J=3$ & $J=4$ & $J=5$ \\
    \hline
    \texttt{SPY1} & 0.44 & 0.40 & 0.36 & 0.32 & 0.32  \\
    \texttt{SPY2} & 0.56 & 0.52  & 0.48 & 0.44 &  0.36 \\
    \texttt{SPY3} & 0.44 & 0.36 & 0.32 & 0.28 & 0.24 \\
    \texttt{MFP1} & 0.52 & 0.48 & 0.44 & 0.40 & 0.32 \\
    \texttt{MFP2} & 0.56 & 0.52  & 0.44 & 0.40 & 0.32 \\
    \toprule
    \end{tabular}
  \caption{Values of $W$ output by our search procedure and used in the experiments associated with Table~\ref{table:K-adaptability-Discrimination}.}
  \label{table:W}
\end{table}


\textbf{Head-to-Head Comparison with Table~\ref{table:DC-Greedy-Discrimination}.} We conduct a head-to-head comparison of our approach with the results from Table~\ref{table:DC-Greedy-Discrimination} which motivated our work. The results are summarized in Table~\ref{table:HEA-TO-HEAD}. From the table we observe a consistent increase of 8-14\% in worst-case coverage of the worse-off group. For example, in \texttt{SPY3}, the coverage of Hispanics has increased from 33\% to 44\%. We can also see that the PoF is moderate, ranging from 1-4.2\%.  The result for the \texttt{MFP1} network suggests a $36\%$ increase in the coverage of the ``Other'' group. We note that, by construction, this group consists of racial minorities with a population less than 10\% of the network size. While this increase has impacted the coverage of ``majority'' groups, the worst-case coverage of the worse-off group has increased by 14\% with a negligible PoF of 2.6\%.

\begin{table}[ht!] 
  \small
  \centering
  \begin{tabular}{r*{8}{c}}
    \toprule
    \multirow{2}{*}{Network Name} & \multirow{2}{*}{Network Size ($N$)} &  \multicolumn{5}{c}{Worst-case coverage of individuals by racial group (\%)} & \multirow{2}{*}{PoF (\%)}\\
    \cmidrule(lr){3-7}
     &  & White & Black & Hispanic & Mixed  & Other  &  \\
    \midrule
    \texttt{SPY1} & 95 & 65 (70) & \textbf{45} (36) & -- & 79 (86) & 88 (94) & 3.3 \\
    \texttt{SPY2} & 117 &{81} {(78)} & -- & \textbf{50} (42) & {72} {(76)} & {73} {(67)} & {1.0} \\
    \texttt{SPY3} & 118 & 90 (88) & -- & \textbf{44} (33) & 85 (95) & 87 (69) & 4.2\\
    \texttt{MFP1} & 165  & 85 (96) & 69 (77) & 42 (69) & 73 (73) & \textbf{64} (28) &  2.6 \\
    \texttt{MFP2} & 182 & \textbf{56} (44) & 80 (85) & 70 (70) & 71 (77)& 72 (72) & 3.4 \\
  \bottomrule \\
  \end{tabular}
  \caption{Reduction in racial discrimination in node coverage resulting from applying our proposed algorithm relative to that of~\cite{tzoumas2017resilient} on the five real-world social networks from Table~\ref{table:racial_composition}, when 1/3 of nodes (individuals) can be selected as monitors, out of which at most 10\% may fail. The numbers correspond to the worst-case percentage of covered nodes across all monitor availability scenarios. The numbers in the parentheses are solutions to the state-of-the-art algorithm~\cite{tzoumas2017resilient} (same numbers as in Table~\ref{table:DC-Greedy-Discrimination}.}
  \label{table:HEA-TO-HEAD}
\end{table}

\section{Supplemental Material: Proof of Statements in Section~\ref{sec:problem_formulation}}\label{apd:np-hardness}
\begin{proof}[Proof of Lemma~\ref{lem:np-hard}]
For the special case when all monitors are available ($\Xi = \{{\textbf{e}}\}$), there is a single community ($C = 1$), and no fairness constraints are imposed ($W=0$), Problem~\eqref{prob:single_stage_final} reduces to the maximum coverage problem, which is known to be $\mathcal{NP}$-hard~\cite{feige1998threshold}.  
\end{proof}


\section{Supplemental Material: Proofs of Statements in Section~\ref{sec:PoF}}
\newar{In all of our analysis, we assume the graphs are undirected. This can be done without loss of generality and the results hold for directed graphs.}
\begin{figure}[ht!]
\centering
\subfloat[Original Graph~\label{fig:a}]{\includegraphics[width= 4.2cm]{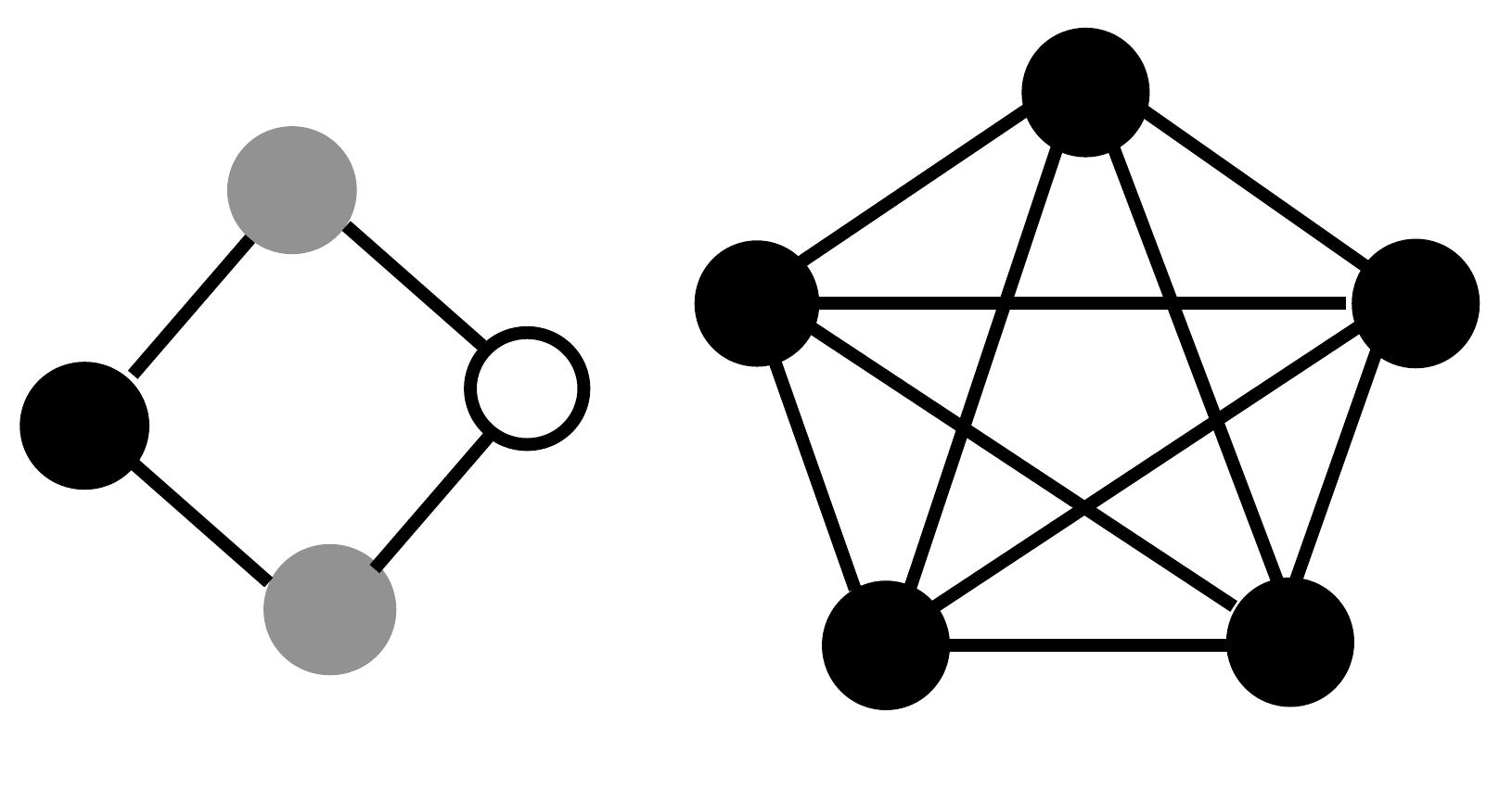}}
\hfill
\subfloat[With fairness~\label{fig:b}]{\includegraphics[width= 4.2cm]{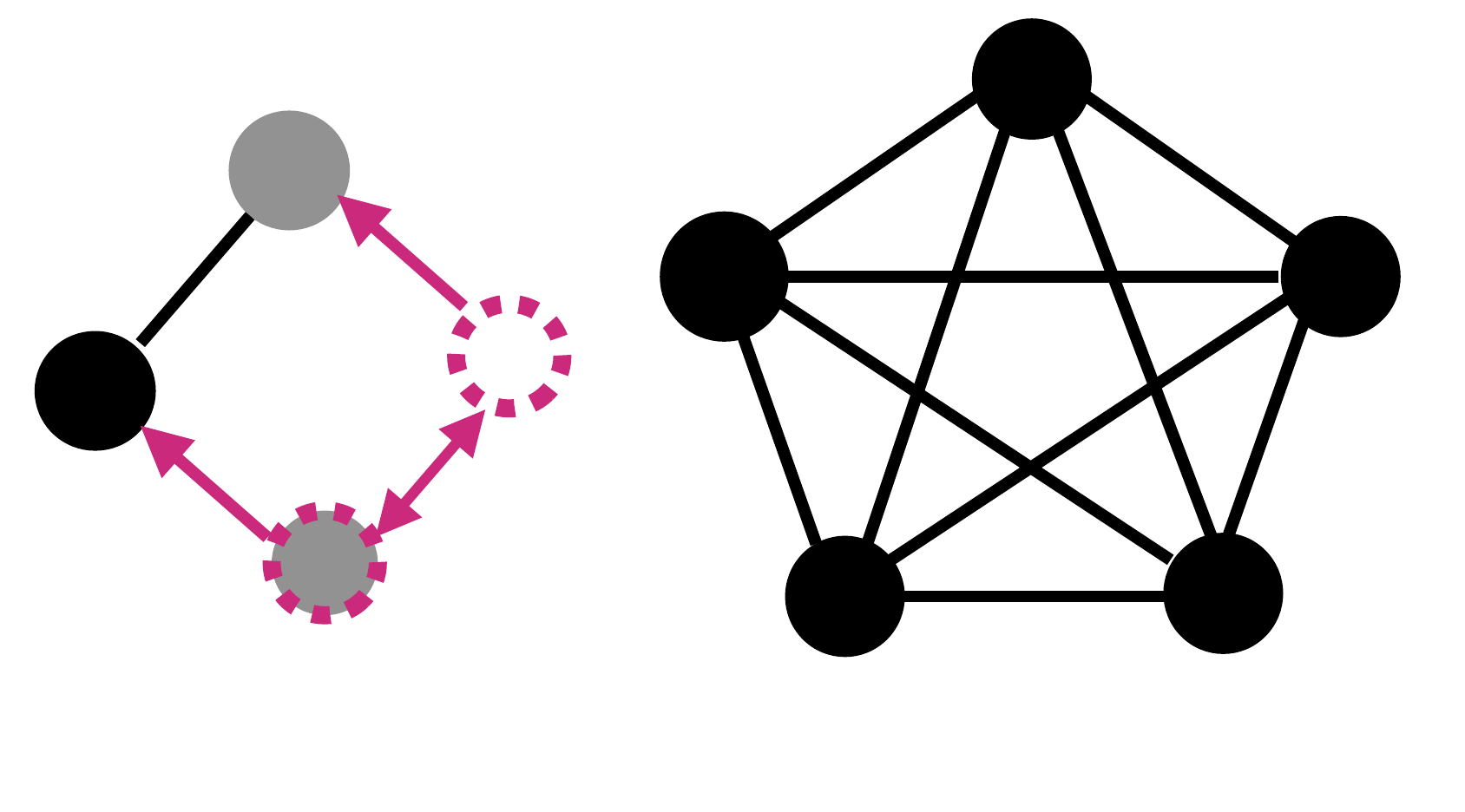}}
\hfill
\subfloat[Without fairness~\label{fig:c}]{\includegraphics[width= 4.2cm]{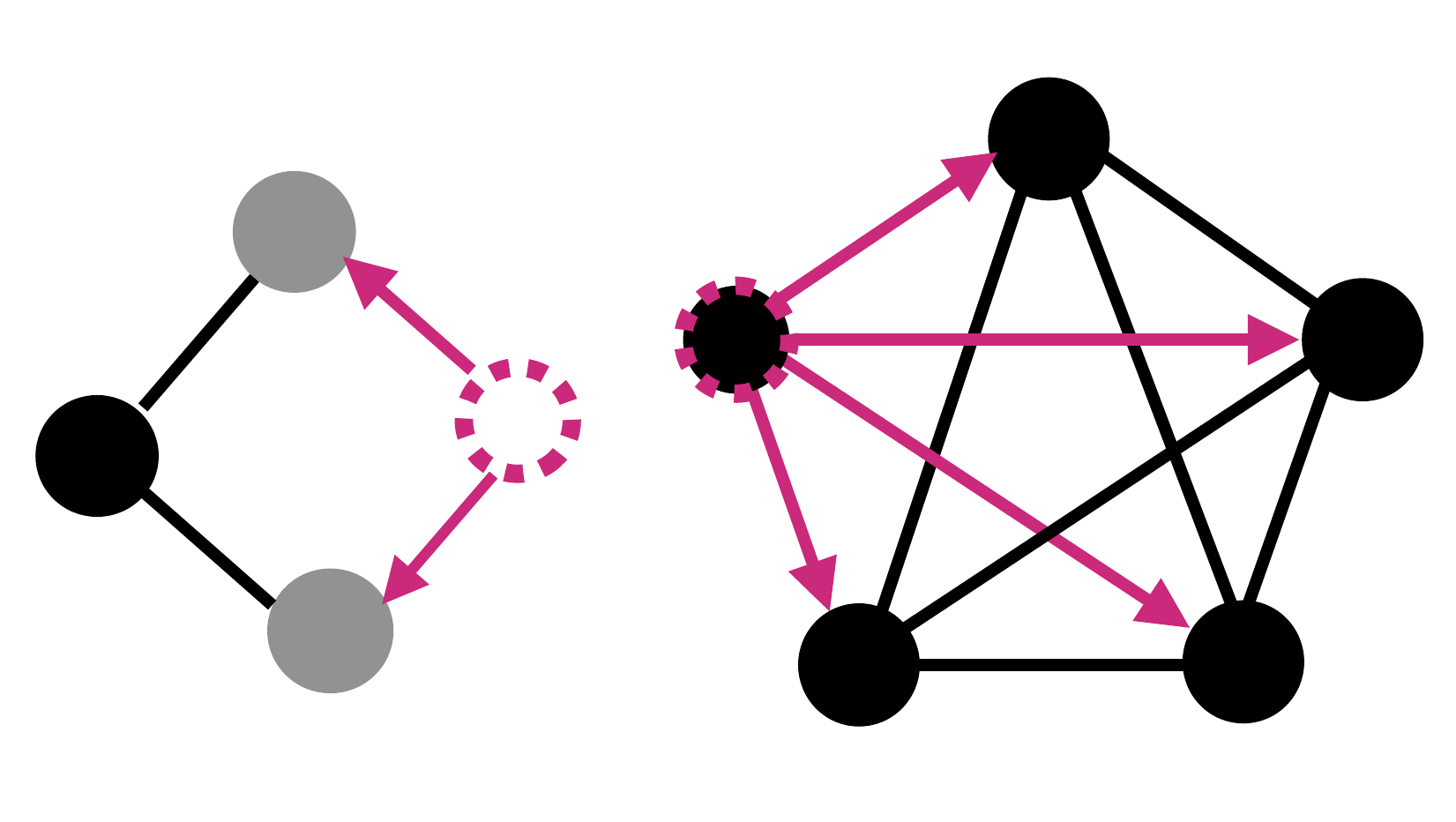}}
\caption{Companion figure to Lemma~\ref{lem:worst-case}. The figures illustrate a network sequence $\{\mathcal G_N\}_{N=5}^\infty$ parameterized by~$N$ and consisting of two disconnected clusters: a small and a large one, with 4 and $N-4$ nodes, respectively. The small cluster remains intact as $N$ grows. The nodes in the large cluster form a clique. In the figures, each color (white, grey, black) represents a different group and we investigate the price of imposing fairness across these groups. The subfigures show the original graph (a) and an optimal solution when $I=2$ monitors can be selected in the cases (b) when fairness constraints are not imposed and (c) when fairness constraints are imposed, respectively. It holds that $\text{\rm{OPT}}^{\text{\rm{fair}}}(\mathcal G_N, 2, 0) = 4$ and $\text{\rm{OPT}}(\mathcal G_N, 2, 0) = N-3$ so that the PoF in $\mathcal G_N$ converges to one as $N$ tends to infinity.}
\label{fig:pof-general-graphs}
\end{figure}

\subsection{Worst-Case PoF}

\begin{proof}[Proof of Lemma~\ref{lem:worst-case}]
Let $\{ \mathcal G_N \}_{N=5}^\infty$ denote the graph sequence shown in Figure~\ref{fig:pof-general-graphs}(a) (wherein all edges are bidirectional). The network consists of three groups (e.g., racial groups) for which fair treatment is important. Network $\mathcal G_N$ {consists} of two disjoint clusters: one {involving} four nodes and a bigger clique containing the remaining $(N-4)$ nodes. Suppose that we can choose $I = 2$ nodes as monitors and that all of them are available ($J=0$). Observe that Problem~\eqref{prob:single_stage_final} is feasible only if $0\leq W \leq (N-3)^{-1}$. For $\textstyle W = (N-3)^{-1},$ the optimal solution places both nodes in the smaller cluster, see Figure~\ref{fig:pof-general-graphs}(b). This way, at least one node from each group is covered. The total coverage for the fair solution is then equal to $\text{\rm{OPT}}^{\text{\rm{fair}}}(\mathcal G_N, 2, 0) = 4$. The maximum achievable coverage under no fairness constraints, however, is obtained by placing one monitor in each cluster, see Figure~\ref{fig:pof-general-graphs}(c). Thus, the total coverage is equal to ${\text{\rm{OPT}}(\mathcal G_N, 2, 0) = N-3}$. As a result, $\textstyle {\rm{PoF}}(\mathcal G_N, 2,0) = 1-4(N-3)^{-1}$ and for $N \geq 4/\epsilon + 3$, it holds that $\textstyle {\rm{PoF}}(\mathcal G_N, 2,0) \geq 1-\epsilon$. The proof is complete.
\end{proof}

\subsection{Supporting Results for the PoF Derivation}

In this section, we provide the preliminary results needed in the derivation of the PoF for both the deterministic and robust graph covering problems. First, we provide two results (Lemmas~\ref{lem:max-degree} and~\ref{lem:I}) from the literature which characterize the maximum degree, as well as the expected number of {maximum}-degree nodes in sparse \newar{Erd\H{o}s R\'{e}ny graphs~\cite{erd6s1960evolution,gilbert1959random}}. We note that in SBM graphs which are used in our PoF analysis, each community $c \in \mathcal C$, when viewed in isolation, is an instance of the Erd\H{o}s R\'{e}nyi graph, in which each edge exists independently with probability $p_c^{\text{in}}$. These results are useful to evaluate the coverage of each community $c\in \mathcal C$ under the sparsity Assumption~\ref{ass:pin}. Specifically, they enable us to show in Lemma~\ref{lem:coverage-in} that, in sparse Erd\H{o}s R\'{e}nyi graphs, the coverage can be evaluated approximately as the sum of the degrees of the monitoring nodes. Thus, the maximum coverage within each community in an SBM network can obtained by selecting the maximum degree nodes. Lastly, we prove Lemma~\ref{lem:coverage-out} which will be useful to show that coverage from monitoring nodes in other communities in SBM networks is negligible.

In what follows, we use $\mathbb G_{N,p}$ to denote \newar{a} random \newar{instance} of Erd\H{o}s R\'{e}ny graphs on vertex set $\mathcal N (= \{1,\ldots,N\})$, where each edge occurs independently with probability $p$.
Following the notational conventions in~\cite{frieze-karonski-2015}, we will say that a sequence of events $\{ \mathbb A_{n} \}_{n=1}^N$ occurs with high probability if $\textstyle \lim_{n\rightarrow \infty} \mathbb P(\mathbb A_{n}) = 1$ and, given a graph $\mathcal G$, we let $\Delta(\mathcal G)$, the maximum degree of vertices of $\mathcal G$.

\begin{theorem}[{\cite[Theorem 3.4]{frieze-karonski-2015}}]
Let $\{\mathbb G_{N,p} \}_{N=1}^\infty$ a sequence of graphs.  
If $p = \newar{\textstyle \Theta(N^{-1})}$, then with high probability
$$
\lim_{N \rightarrow \infty} \Delta(\mathbb G_{N,p}) =  \newar{\frac{\log N}{\log\log N}}.
$$
\label{lem:max-degree}
\end{theorem}


\begin{lemma}
Let $\{\mathbb G_{N,p} \}_{N=1}^\infty$ a sequence of graphs with $p = \newar{\textstyle \Theta(N^{-1})}$. Let \newar{$\textstyle \sigma(N) := \log N({\log\log N})^{-1}$. Then, it holds that}
\begin{equation*}
\newar{\EX[X_{\sigma(N)} (\mathbb G_{N,p})] \geq \displaystyle N^{\frac{\log \log \log N - o(1)}{\log \log N}},}
\label{equ:order-I}
\end{equation*}
\label{lem:I}
\end{lemma}
where \newar{$X_{\sigma(N)}(\mathbb G_{N,p})$ is the number of vertices of degree $\sigma(N)$} in $\mathbb G_{N,p}$. 
\begin{proof}
We borrow results from~\cite[Theorem 3.4]{frieze-karonski-2015}, where the authors show that 
\begin{equation*}
\newar{\EX[X_{\sigma(N)} (\mathbb G_{N,p})] = \exp\left(\frac{\log N}{\log \log N} \left(\log \log \log N - o(1)\right)+O\left(\frac{\log N}{\log\log N + 2\log \log\log N}\right)\right)},
\end{equation*}

We further simplify the expression in Lemma~\ref{lem:I} by eliminating the $O(.)$ term and we obtain
\begin{equation*}
\newar{\EX[X_{\sigma(N)} (\mathbb G_{N,p})] \geq N^{\frac{\log \log \log N - o(1)}{\log \log N}},}
\label{equ:order-I}
\end{equation*}
\end{proof}

Lemma~\ref{lem:I} ensures that our budget for selecting monitors $I=O(\log N)$, is (asymptotically) smaller than number of nodes with degree $\Delta(\mathbb{G}_{N,p}).$
\begin{lemma}
Let $\{ \mathbb G_{N,p} \}_{N=1}^\infty$ be a sequence of graphs with $\textstyle p = \Theta(N^{-1})$. Suppose \newar{that the number of monitors is $I=O(\log N)$.}
Then, \newar{for all $\nu$, there exists a graph $\mathbb G_{N,p}$ such that }the \newar{difference between the expected maximum coverage in $\mathbb G_{N,p}$  
and the expected number of neighbors of the monitoring nodes is bounded. Precisely, if ${\bm x}(\mathbb G_{N,p})$ is the indicator vector of the highest degree nodes in~$\mathbb G_{N,p}$, we have}
$$
\displaystyle \sum_{n \in \mathcal N} \EX \left[ \newar{\bm x_n(\mathbb G_{N,p})} | \delta_{{\mathbb G}_{N,p}}(n)| \right] -  \EX \left[F_{{\mathbb G}_{N,p}}(\bm x(\mathbb G_{N,p}),\emph{\textbf{{{e}}})} \right] \; \leq \; \nu,
$$
where $\delta_{\mathbb G_{N,p}}(n)$ is the set of neighbors of $n$ in $\mathbb G_{N,p}$ and \newar{$\nu$ is the error term and it is $\textstyle \nu = o(1)$.} 
\label{lem:coverage-in}
\end{lemma}

\begin{proof}
Let $Y_{n}$ be the event that node $n$ is covered. Also, let $Z^{i}_{n}$ the event that node $n$ is covered by the \newar{$i$th} highest degree node 
(and by potentially other nodes too). Without loss of generality, assume that the nodes with lower indexes have higher degrees, i.e., $\textstyle |\delta(1)|\geq \cdots \geq |\delta(N)|$. The probability that node $n$ is covered can be written as 
\begin{equation}
\displaystyle \mathbb P(Y_{n}) = \mathbb P \left(\cup^{I}_{i = 1}Z^{i}_{n} \right).
\label{equ:union-probability}
\end{equation}

From the Bonferroni inequalities, we have 
\begin{equation}
\mathbb P(\cup^{I}_{i = 1}Z^{i}_{n}) \displaystyle \; \geq \;  \sum_{i=1}^{I}  \left(\mathbb P(Z^{i}_{n}) - \sum_{j = i}^{ I}\mathbb P(Z^{i}_{n} \cap Z^{j}_{n})\right)
\label{equ:prob-lower-bound}
\end{equation}
and
\begin{equation}
\mathbb P(\cup^{I}_{i = 1}Z^{i}_{n}) \; \leq \; \sum_{i=1}^{I} \mathbb P(Z^{i}_{n}).
\label{equ:prob-upper-bound}
\end{equation}

Define $\textstyle Y:=\sum_{i=1}^N Y_n$ as the (random) total coverage. With a slight abuse of notation, we view $Y_n$ and $Z^{i}_{n}$ as Bernoulli random binary variables that are equal to 1 if and only if the associated event occurs. 
As a result, we can substitute the probability terms with their expected values. 
%
Combining Equations~\eqref{equ:union-probability},~\eqref{equ:prob-lower-bound} and~\eqref{equ:prob-upper-bound}, \newpv{we obtain}
\begin{equation*}
\displaystyle \sum_{i=1}^{I} \left(\EX[Z^{i}_{n}] - \sum_{j = i}^{I}\EX[Z^{i}_{n} Z^{j}_{n}] \displaystyle \right) \; \leq \; \EX[Y_{n}] \; \leq \; \sum_{i=1}^{I} \EX[Z^{i}_{n}], \quad \forall n\in\mathcal N,
\end{equation*}
where we used the fact that $\mathbb P(Z^{i}_{n} \cap Z^{j}_{n}) = \mathbb P(Z^{i}_{n}) \mathbb P(Z^{j}_{n}) = \mathbb E(Z^{i}_{n}) \mathbb E(Z^{j}_{n}) = \mathbb E(Z^{i}_{n} Z^{j}_{n})$ by independence of the events $Z^{i}_{n}$ and $Z^{j}_{n}$.
%
Summing over all $n$ yields
\begin{equation*}
\displaystyle \sum_{n \in \mathcal N}\left(\sum_{i=1}^{I} \EX[Z^{i}_{n}] - \sum_{j=i}^{I}\EX[Z^{i}_{n}Z^{j}_{n}]\right) \displaystyle \; \leq  \;  \sum_{n \in \mathcal N}\EX[Y_n] \; \leq \; \sum_{n \in \mathcal N}\sum_{i=1}^{I} \EX[Z^{i}_{n}].
\end{equation*}
Changing the order of the summations, it follows that
\begin{equation*}
\displaystyle \sum^{I}_{i=1}\left(\sum_{n \in \mathcal N} \EX[Z^{i}_{n}] - \sum_{j=i}^{I}\sum_{n\in\mathcal N}\EX[Z^{i}_{n}Z^{j}_{n}]\right) \displaystyle \; \leq \; \EX[Y] \; \leq \; \sum_{i=1}^{I} \sum_{n \in \mathcal N}\EX[Z^{i}_{n}],
\end{equation*}
where we have used $\textstyle \mathbb E [ Y ]=\sum_{i=1}^N \mathbb E [ Y_n ]$. By definition of $\delta_{\mathbb G_{N,p}}(i)$, since $x_i(\mathbb G_{N,p})=1$ for $i=1,\ldots,I$, it holds that the number of nodes covered by node $i$, $\textstyle \sum_{n\in\mathcal N}\EX[Z^{i}_{n}] = \EX[|\delta_{\mathbb G_{N,p}}(i)|]$. 
Also, we remark that $\textstyle \EX[Y]= \EX[ F_{\mathbb G_{N,p}}(\bm x(\mathbb G_{N,p}), \textbf{{\rm{e}}})]$. Thus, the above sequence of inequalities is equivalent to
\begin{equation*}
\displaystyle \sum_{i = 1}^{I}\left( \EX[|\delta_{\mathbb G_{N,p}}(i)|] - \sum_{j=i}^{I}\sum_{n\in\mathcal N}\EX[Z^{i}_{n}Z^{j}_{n}] \right)\displaystyle \; \leq \;  \EX[ F_{\mathbb G_{N,p}}(\bm x(\mathbb G_{N,p}), \textbf{{{e}}})] \; \leq \; \sum_{i=1}^{I} \EX[|\delta_{\mathbb G_{N,p}}(i)|],
\end{equation*}
where, by reordering terms, we obtain
\begin{equation*}
0  \; \leq  \; \displaystyle \sum_{i = 1}^{I} \EX[|\delta_{\mathbb G_{N,p}}(i)|] - \EX[ F_{\mathbb G_{N,p}}(\bm x(\mathbb G_{N,p}), \textbf{{{e}}})] \; \leq \; \sum_{i=1}^{I}\sum_{j=i}^{I}\sum_{n\in\mathcal N}\EX[Z^{i}_{n}Z^{j}_{n}]. 
\end{equation*}
Note that $\newar{\EX \left[\bm x_n(\mathbb G_{N,p})\right] = 1, \forall n\leq I}$ since by assumption the nodes are ordered by decreasing order of their degree, so the nodes indexed from 1 to $I$ are selected in each realization of the graph. Thus,
$$
\begin{array}{rcl}
\displaystyle \sum_{i = 1}^{I} \EX[|\delta_{\mathbb G_{N,p}}(i)|] 
& = & \displaystyle \sum_{n \in \mathcal N} \EX \left[ \newar{\bm x_n(\mathbb G_{N,p})}\right] \EX \left[ | \delta_{{\mathbb G}_{N,p}}(n)| \right] \\
& = & \displaystyle \sum_{n \in \mathcal N} \EX \left[\newar{\bm x_n(\mathbb G_{N,p})}| \delta_{{\mathbb G}_{N,p}}(n)| \right],
\end{array}
$$
which yields
\begin{equation}
\displaystyle \displaystyle \sum_{n \in \mathcal N} \EX \left[\newar{\bm x_n(\mathbb G_{N,p})}| \delta_{{\mathbb G}_{N,p}}(n)| \right] - \EX[ F_{\mathbb G_{N,p}}(\bm x(\mathbb G_{N,p}), \textbf{{{e}}})] \; \leq \; \sum_{i=1}^{I}\sum_{j=i}^{I}\sum_{n\in\mathcal N}\EX[Z^{i}_{n}Z^{j}_{n}]. 
\label{equ:error-bound}
\end{equation}
%
%
The right-hand side of Equation~\eqref{equ:error-bound} is the error term and we denote it by $\textstyle \nu = \sum_{i=1}^{I}\sum_{j=i}^{I}\sum_{n\in\mathcal N}\EX[Z^{i}_{n}Z^{j}_{n}]$. This error term determines the difference between the true value of the coverage and the expected sum of the degrees of the monitoring nodes. Given that $\textstyle p = \Theta({N}^{-1})$, we can precisely evaluate the error term. First, we note that since in the Erd\H{o}s-R\'{e}nyi model edges are drawn independently, we can write $\EX[Z^{i}_{n}Z^{j}_{n}] = \EX[Z^{i}_{n}]\EX[Z^{j}_{n}]$. Using \newar{Theorem~\ref{lem:max-degree} and Lemma~\ref{lem:I}}, and given that the monitors are the highest degree nodes in any realization of the graph, we can write 
$$
\EX[Z^{i}_{n}] = \EX[Z^{j}_{n}]  = \Theta\left(\frac{1}{N}\frac{\log N}{\log\log N}\right).
$$
We thus obtain 
$$\textstyle \nu = \textstyle \Theta\left(\frac{I^2}{N}\left(\frac{\log N}{\log\log N}\right)^2\right).$$ 
By the assumption on the order of $I$, it follows that $\textstyle\lim_{N \rightarrow \infty} \nu = 0$, which concludes the proof.
\end{proof}




We now prove the following lemma which will be used in proof of the subsequent results.
\newar{\begin{lemma}
Let $X_i$ for $i=1,\dots,Q$ be $Q$ i.i.d samples from normal distribution with mean $\mu$ and standard deviation $\sigma$. Also, let $\textstyle Z = \max_{i\in\{1,\cdots, Q\}} X_i.$ It holds that
$$
\EX[Z] \leq \mu + \sigma \sqrt{2 \log Q}.
$$
\label{lem:normal-approx}
\end{lemma}
\begin{proof}
By Jensen's inequality,
\begin{equation*}
    \begin{array}{ccl}
    \exp(t\EX[Z]) \leq \EX[\exp(tZ)] &  =  &  \EX[\exp(t\max_{i=1,\dots, Q} X_i)] \\
    & \leq & \sum^{Q}_{i=1} \EX[\exp(t X_{i})] \\
    & = & Q \exp(\mu t + t^2\sigma^2 /2),
    \end{array}
\end{equation*}
where the last equality follows from the definition of the Gaussian moment generating function. 
Taking the logarithm of both sides of this inequality, we can obtain
$$
\EX[Z] \leq \mu + \frac{\log Q}{t} + \frac{t\sigma^2 }{2}.
$$
For the tightest upper-bound, we set $t= \sqrt{2\log Q} / \sigma$. Thus, we obtain
$$
\EX[Z] \leq \mu + \sigma \sqrt{2 \log Q}.
$$
\end{proof}
}

\begin{lemma}
Consider $\mathbb{B}_{N,M,p}$ to be a random instance of a bipartite graph on the vertex set $\mathcal N = \mathcal L \cup \mathcal R$, where \newar{$N = |\mathcal R \cup \mathcal L|$ and $M := |\mathcal R|$ and $p= O\left((M \log^2 M)^{-1}\right)$ is the probability that each edge exists (independently).} Suppose that monitoring nodes can only be chosen from the set $\mathcal L$ and that at most $I$ monitors can be selected. Then, it holds that 
\begin{equation*}
 \EX \left[ \max_{ \begin{smallmatrix}
  \bm x \in \mathcal \{0,1\}^{|\mathcal L|}: \\ \sum_{n\in\mathcal L}{\bm x_{n} = I}
 \end{smallmatrix}}F_{\mathbb B_{N,M,p}}(\bm x, \emph{\textbf{e}}) \right] = 
 I O\left(\frac{1}{{\log^{2}{M}}}\right).
\end{equation*}
\label{lem:coverage-out}
\end{lemma}
\begin{proof}
We note that the degree of node $i$, $\delta_{ \newar{\mathbb B_{N,M,p}}}(i)$, follows a binomial distribution with mean $Mp$. Given we are interested in $N, M\rightarrow \infty$, we can approximate the binomial distribution with a normal distribution~\cite{walrand2004lecture} with mean $Mp$ and standard deviation $\sqrt{Mp(1-p)}$. Using the result of Lemma~\ref{lem:normal-approx}, we obtain
$$
\EX[\Delta_{\mathbb B_{N,M,p}}] = \newar{O}\left(Mp + \sqrt{Mp(1-p)}\sqrt{2\log{(\newar{N}-M)}}\right) = \newar{O}(M p).
$$
Using the above result combined with the assumption on $p$, we can bound the expected maximum degree of $\newar{\mathcal B}$. 
\begin{equation*}
    \EX[\Delta_{\mathbb B_{N,M,p}}] =
    \newar{O\left(\frac{1}{{\log^{2}{M}}}\right)}.
\end{equation*}
As a result, the maximum expected coverage of the $I$ monitoring nodes is upper-bounded as
\begin{equation*}
    \EX \left[ \max_{ \begin{smallmatrix}
  \bm x \in \{0,1\}^{N}: \\ \sum_{n\in\mathcal L}{x_{n} = I}
 \end{smallmatrix}}F_{\mathbb B_{N,M,p}}(\bm x,\textbf{{\rm{e}}}) \right] \; \leq \; I\EX[\Delta_{\mathbb B_{N,M,p}}]  \; = \; 
 \newar{
 IO\left(\frac{1}{\log^{2}{M}}\right)
 }.
\end{equation*}
and the proof is complete.
\end{proof}



\subsection{PoF in the Deterministic Case}

Next, we prove the main result which is the derivation of the PoF for the deterministic graph covering problem. The idea of the proof is as follows: by Lemmas~\ref{lem:I}  and~\ref{lem:coverage-in}, we are able to evaluate the coverage of each community. By Lemma~\ref{lem:coverage-out}, we upper bound the between-community coverage. \newar{In other words, based on Lemma~\ref{lem:coverage-out}, we conclude that in every instance of the coverage problem, the between-community coverage is zero (asymptotically) with high probability. Thus, the allocation of monitoring nodes is only dependant on the within-community coverage.} Using this observation, we can determine the allocation of the monitors both in the presence and absence of fairness constraints. Subsequently, we are able to evaluate the coverage in both cases. PoF can be then computed based on these two quantities, see Equation~\eqref{equ:POF_average}.

\begin{proof}[Proof of Proposition~\ref{prop:PoF}]
\newar{Let $\mathbb S_N$ be a random instance of the SBM network with size $N$. }
Consider $\textstyle \bm s(\mathbb S_N) \in \mathbb{Z}^{C}$ to be the number of allocated monitoring nodes to each of the $C$ communities, i.e., $\textstyle \bm s_{c}(\mathbb S_N) = \sum_{n\in\mathcal N_{c}}\bm x_{n}(\mathbb S_N).$
Using the result of Lemmas~\ref{lem:coverage-in} and ~\ref{lem:coverage-out}, we can measure the expected maximum coverage as
\newar
{
\begin{equation*}
\lim_{N \rightarrow \infty} \EX[\text{OPT}(\mathbb S_N, I, 0)] = \lim_{N\rightarrow \infty} \EX \left[\max_{
  \bm x(\mathbb S_N) \in \mathcal X}
 F_{\mathbb S_N}(\bm x,{\textbf{e}}) \right] =  \EX \left[\lim_{N\rightarrow \infty} \max_{
  \bm x(\mathbb S_N) \in \mathcal X}
 F_{\mathbb S_N}(\bm x,{\textbf{e}}) \right],
\end{equation*}
where the last equality is obtained by exchanging the expectation and limit. 
Using Lemma~\ref{lem:max-degree} and since the maximum degree is convergent to $d(c)$, we can exchange the limit and maximization term. Thus, we will have
\begin{equation*}
\begin{array}{ccl}
\EX \left[\lim_{N\rightarrow \infty} \max_{
  \bm x(\mathbb S_N) \in \mathcal X}
 F_{\mathbb S_N}(\bm x, {\textbf{e}}) \right] & = & \EX \left[ \max_{\bm x(\mathbb S_N) \in \mathcal X}\lim_{N\rightarrow \infty}
 F_{\mathbb S_N}(\bm x, {\textbf{e}}) \right] 
 \\ 
 & = & \EX \left [\max_{\bm {s}(\mathbb S_N) \in \mathbb Z^C} \sum_{c\in\mathcal C}{{\bm s}_{c}(\mathbb S_N)d(c) + o(1)}\right],
\end{array}
\end{equation*}
}
\newar{which given that $d(c)$ is only \newpv{dependent} on the size of the communities in $\mathbb S_{N}$ is equivalent to}
\begin{equation}
 \lim_{N\rightarrow \infty}\EX[\text{OPT}(\mathbb S_N, I, 0)] =   \max_{\bm {s}(\mathbb S_N)} \sum_{c\in \mathcal C}{\bm{s}_{c}(\mathbb S_N)d(c) + o(1)}. 
\label{equ:OPT-TOTAL}   
\end{equation}
\newar{Equation~\eqref{equ:OPT-TOTAL} suggests that for large enough $N$, the maximum coverage is only dependent on the \emph{number} of the monitoring nodes allocated to each community. Also, the allocation is the same for all random instances so we can drop the dependence of $\bm s$ on $\mathbb S_{N}$.}  In right-hand side of Equation~\eqref{equ:OPT-TOTAL}, the first term is the within-community (Lemma~\ref{lem:coverage-in}), and the second term is the between-community (Lemma~\ref{lem:coverage-out}) coverage. 

\newar{In the analysis below, all the evaluations are for large enough $N$. Therefore, we drop the $\textstyle \lim_{N\rightarrow \infty}$ for ease of notation.} According to Equation~\eqref{equ:OPT-TOTAL} the between-community coverage is negligible, compared to the within-community coverage. 
This suggests that the maximum achievable coverage will be obtained by placing all the monitoring nodes in the largest community, with the largest value of $d(c)$, where the assumption on $I$, as given in the premise of the proposition, combined with Lemma~\ref{lem:I} guarantee that such a selection is possible. Thus, we obtain
$$
\EX[\text{OPT}(\mathbb S_N, I, 0)] = I d(C) + o(1).
$$
Next, we measure $\EX[\text{OPT}^{\text{fair}}(.)]$, where in addition to optimization problem in Equation~\eqref{equ:OPT-TOTAL}, the allocation is further restricted to satisfy all the fairness constraints. 
\begin{equation}
     \frac{s_{c}}{|\mathcal N_c|}d(c)+ o(1) \geq W \quad \forall c \in \mathcal C,
    \label{equ:u-norm-absolute}
\end{equation}
in which, $o(1)$ is the term that compensates for the coverage of the nodes in other communities, and is small due to the regimes of $\textstyle p_{\newar{cc'}}^{\text{out}}, \;\forall c,c'\in \mathcal C $ and the budget $I.$ At optimality and for the maximum value of $W$, we have 
$$
 \left|{s_{c}}{|\mathcal N_c|}^{-1}d(c) - {s_{c'}}{|\mathcal N_{c'}|}^{-1}d(c')\right| \leq \delta \;\; \forall c, c' \in \mathcal C, \delta \leq \left|{d(1)}{|\mathcal N_1|}^{-1} -{d(C)}{|\mathcal N_C|}^{-1} \right|.
$$
This holds because otherwise one can remove on node from the group with higher value of $\textstyle {s_{c}}{|\mathcal N_c|}^{-1}d(c)$ to a group with less value and thus increase the normalized coverage of the worse-off group and this contradicts the fact that $W$ is the maximum possible value. This suggests that in a fair solution, the normalized coverage is \textit{almost} equal across different groups, given that $\textstyle \lim_{N \rightarrow \infty} \delta = 0$. As a result, the monitoring nodes should be such that
\begin{equation*}
W \leq \frac{s_{c}}{|\mathcal N_c|}d(c) + o(1) \leq W + \delta, \; \forall c \in\mathcal C.
\label{equ:allocation-0}
\end{equation*}
From this, it follows that
\begin{equation}
W - o(1)\leq \frac{s_{c}}{|\mathcal N_c|}d(c) \leq W + o(1).
\label{equ:allocation-1}
\end{equation}
By assumption, there must be an integral $s_c$ that satisfies the above relation. Note that if we could relax the integrality assumption, $\textstyle s_c = {W|\mathcal N_c|}{d(c)^{-1}}$. Due to the integrality constraint, and according to Equation~\eqref{equ:allocation-1}, we set $\textstyle {s_{c}}{|\mathcal N_c|}^{-1}d(c) = W + o(1) $, where the o(1) term is to account for the discretizing error, which results in $\textstyle s_{c} = {W|\mathcal N_c|}{d(c)^{-1}} + O(1)$, where $\textstyle O(1) \leq 1$ (As we can not make a higher error in rounding). Also, since $\textstyle \sum_{c\in\mathcal C} s_c = I$, we can obtain the value of $W$ as
$$
\displaystyle W =  \frac{I}{\sum_{c\in\mathcal C}\frac{|\mathcal N_c|}{d(c)}}+o(1).
$$

As a result
\begin{equation*}{s}_{c} = \frac{I}{\sum_{c \in \mathcal C}{\frac{|\mathcal N_c|}{d(c)}}}  \frac{|\mathcal N_c|}{d(c)} + O(1) \quad \forall c \in \mathcal C.
\label{equ:allocation_6}
\end{equation*}
\newar{We now define~$\textstyle \kappa := {I}\left({\sum_{c\in\mathcal C}\frac{|\mathcal N_c|}{d(c)}}\right)^{-1}$ for a compact representation.}

So far, we obtained the allocation of the monitoring nodes to satisfy the fairness constraints. This is enough to evaluate the coverage under the fairness constraints. Now, we can evaluate the PoF as defined by Equation~\eqref{equ:POF_average}. 
\begin{equation*}
\renewcommand{\arraystretch}{1.8}
\begin{array}{crcl}
& \EX[\text{OPT}(\mathbb S_N, I, 0)] & = &I d(C) \vspace{0.2cm}  \\ 
\Rightarrow &  \displaystyle - \frac{1}{\EX[\text{OPT}(\mathbb S_N, I, 0)]} & =& -\frac{1}{I \,d(C)}  \vspace{0.2cm} \\  
\Rightarrow   & \displaystyle -\frac{\EX[\text{OPT}^{\text{fair}}(\mathbb S_N, I, 0)]}{\newar{\EX[\text{OPT}(\mathbb S_N, I, 0)}]} & = &- \frac{\kappa  \sum_{c\in\mathcal C }{ \frac{|\mathcal N_c|}{d(c)}} d(c)}{I \,d(C)} \newar{- o(1)} \vspace{0.2cm}  \\ 
\Rightarrow & \displaystyle 1-\frac{\EX[\text{OPT}^{\text{fair}}(\mathbb S_N, I, 0)]}{\EX[\text{OPT}(\mathbb S_N, I, 0)]} & = &1 - \frac{\kappa \sum_{c\in\mathcal C }{  \frac{|\mathcal N_c|}{d(c)}} d(c)}{I\,d(C)} - o(1) \vspace{0.2cm}  \\ 
\Rightarrow & \displaystyle \overline{\rm{PoF}}(I, 0)& = & 1 - \frac{\kappa \sum_{c\in\mathcal C }{  |\mathcal N_c| }}{I\, d(C)} - o(1) \\
\Rightarrow & \displaystyle \overline{\rm{PoF}}(I, 0)& = & 1 - \frac{\sum_{c\in\mathcal C}{|\mathcal N_c|}}{{\sum_{c\in\mathcal C}{|\mathcal N_c|d(C)/d(c)}}} - o(1). 
\end{array}
\end{equation*}
\end{proof}

\subsection{PoF in the Robust Case}
\begin{proof}[Proof of Proposition~\ref{prop:price_of_fairness_uncertain}]
The idea of the proof is similar to Proposition~\ref{prop:PoF}, with the exception that the fair allocation of the monitoring nodes will be affected by the uncertainty. 
Consider $\textstyle \bm s$ to be the number of allocated monitoring nodes to each of the $C$ communities, i.e., $\textstyle s_{c} = \sum_{n\in\mathcal N_{c}}x_{n}.$
Using the result of lemma~\ref{lem:coverage-in}, and ~\ref{lem:coverage-out}, we can measure the expected maximum coverage as
\begin{equation*}
  \EX[\text{OPT}(\mathbb S_N, I, J)] = (I-J)d(c) + o(1).  
\end{equation*}
That is because, in the worst-case $J$ nodes fail, thus only $\textstyle (I-J)$ nodes can cover the graph. Next, we measure $\EX[\text{OPT}^{\text{fair}}(.)]$, where in addition to optimization problem in Equation~\eqref{equ:OPT-TOTAL}, the allocation is further restricted to satisfy all the fairness constraints. Given that at most $J$ nodes may fail, we need to ensure after fairness constraints are satisfied after the removal of $J$ nodes. We momentarily revisit the fairness constraint in the deterministic case. 
\begin{equation*}
     \frac{s_{c}}{|\mathcal N_c|}d(c) + o(1) \geq W \quad \forall c \in \mathcal C,
\end{equation*}
in which, $o(1)$ is the term that compensates for the coverage of the nodes in other communities, and is small due to the regimes of $p^{\text{out}}, $ and the budget $I.$
Under the uncertainty, we need to ensure that these constraints are satisfied even after $J$ nodes are removed. In other words
\begin{equation*}
     \frac{(s_{c}-J)}{|\mathcal N_c|}d(c) + o(1) \geq W \quad \forall c \in \mathcal C.
\end{equation*}
At optimality and for the maximum value of ${W}$, we have
$$
 \left|{(s_{c}-J)}{|\mathcal N_c|}^{-1}d(c) - {(s_{c'}-J)}{|\mathcal N_{c'}|}^{-1}d(c')\right| \leq \delta \;\; \forall c, c' \in \mathcal C, \delta \leq \left|{d(1)}{|\mathcal N_1|}^{-1} -{d(C)}{|\mathcal N_C|}^{-1}\right|.
$$
This holds because otherwise one can remove on node from the group with higher value of $\textstyle {s_{c}}{|\mathcal N_c|}^{-1}d(c)$ to a group with less value and thus increase the normalized coverage of the worse-off group and this contradicts the fact that $W$ is the maximum possible value. 

This suggests that in a fair solution, the normalized coverage is \textit{almost} equal across different groups, given that $\delta \rightarrow 0, \text{as } \mathcal N_c \rightarrow \infty , \forall c\in\mathcal C$. 
Following the proof of Proposition~\ref{prop:PoF}, the discretizing error can be handled by setting $ (s_{c}-J){|\mathcal N_c|}^{-1}d(c) = W + o(1) $, where the o(1) term is to account for the discretizing error. As a result
$$s_{c} = \frac{|\mathcal N_c|W}{d(c)} + J + O(1),$$ where $O(1) \leq 1$ (As we can not make a higher error in rounding). 
This suggests that a fair allocation is the one that places $J$ nodes in each community, regardless of the community size. The remaining monitors are allocated with respect to the relative size of the communities. 

Summing over all $s_c$ and since $\textstyle \sum_{c\in\mathcal C} s_c = I$ we obtain 
$$
\displaystyle W =  \frac{(I-CJ)}{\sum_{c\in\mathcal C}\frac{|\mathcal N_c|}{d(c)}}+o(1).
$$
As a result
\begin{equation*}
s_{c} = \frac{(I-CJ)}{\sum_{c \in \mathcal C}{\frac{|\mathcal N_c|}{d(c)}}}  \frac{|\mathcal N_c|}{d(c)} + \newar{J} +O(1) \quad \forall c \in \mathcal C.
\label{equ:allocation_6}
\end{equation*}


\newar{As defined in the premise of the proposition, }  $\textstyle \eta = {(I-CJ)}\left({\sum_{c\in\mathcal C}\frac{|\mathcal N_c|}{d(c)}}\right)^{-1}$. 

So far, we obtained the allocation of the monitoring nodes, to satisfy the fairness constraints.


\newar{Now, we evaluate the coverage, i.e., objective value of Problem~\eqref{prob:single_stage_final}, under the obtained fair allocation. Since the fairness constraints are satisfied under all the scenarios, the worst-case scenario is the one that results in the maximum loss in the total coverage. This corresponds to the case that $J$ nodes from the largest community $(\mathcal{N}_{C})$ fail. As a result the expected coverage can be obtained by 
$$
\EX[{\rm{OPT}^{fair}}(\mathbb S_N, I, J)] =  \sum_{c\in\mathcal C }\left({\eta \frac{|\mathcal N_c|}{d(c)}} d(c) + Jd(c) + O(1)d(c)\right) - Jd(C).
$$
}

Now, we can evaluate the PoF as defined by Equation~\eqref{equ:POF_average}. 
\begin{equation*}
    \begin{array}{crcl}
& \EX[\text{OPT}(\mathbb S_N, I, J)] & = &(I-J) d(C) \\ \Rightarrow & 
\displaystyle - \frac{1}{\EX[\text{OPT}(\mathbb S_N, I, J)]}  & = & -\frac{1}{(I-J)d(C)} \\ \Rightarrow 
& \displaystyle -\frac{\EX[\text{OPT}^{\text{fair}}(\mathbb S_N, I, J)]}{\EX[\newar{\text{OPT}(\mathbb S_N, I, J)]}}  & = & - \frac{\sum_{c\in\mathcal C }\left(\eta |\mathcal N_c| + J d(c)\right) - J d(C)}{(I-J)d(C)} - \newar{o(1)}\\  \Rightarrow
& \displaystyle 1-\frac{\EX[\text{OPT}^{\text{fair}}(\mathbb S_N, I, J)]}{\EX[\text{OPT}(\mathbb S_N, I, J)]} & = &1 - \frac{\sum_{c\in\mathcal C }{\eta |\mathcal N_c|} \newar{+ \sum_{c\in\mathcal C\setminus\{C\}}J d(c)}}{(I-J)d(C)} - o(1) \\ \Rightarrow 
& \displaystyle \overline{\rm{PoF}}(I, J) & = & 1 - \frac{\sum_{c\in\mathcal C }{\eta |\mathcal N_c| }}{(I - J)d(C)} -\frac{\newar{J \sum_{c\in\mathcal C\setminus\{C\}}d(c)}}{(I - J)d(C)} - o(1). 
\end{array}
\end{equation*}
\end{proof}


\section{Supplemental Material: Proofs of Statements in Section~\ref{sec:solution-approach}}
\label{apd:max-min-max}

\subsection{Equivalent Reformulation as a Max-Min-Max Robust Optimization Problem}
\begin{proof}[Proof of Proposition~\ref{prop:equivalent-formulation}]
Let $\bm \bar{x}$ be feasible in Problem~\eqref{prob:single_stage_final}. It follows that it is also feasible in Problem~\ref{prob:two_stage_final}.
For a fixed $\bm \bar{\xi}$, we show that
\begin{equation*}
\renewcommand{\arraystretch}{1.8}
    \begin{array}{lccl}
    \sum_{c\in\mathcal C} F_{\mathcal G, c}(\bar{\bm {x}}, \bar{\bm {\xi}}) & = &\displaystyle \max_{\bm y}  & \displaystyle \sum_{c\in\mathcal C}\sum_{n\in\mathcal N_c}\bm y_{n}  \\
    & &\text{s.t.}  & \bm y_{n}\leq \sum_{\nu\in\delta(n)}\bar{\bm\xi}_{\nu}\bar{\bm x}_{\nu} \\
    & & & \displaystyle \sum_{n\in\mathcal C}\bm y_{n} \geq W|\mathcal N_c|, \; \forall c\in\mathcal C
    \end{array}
\end{equation*}

Since $\bar{\bm x}$ is feasible in Problem~\eqref{prob:single_stage_final}, it holds that 

\begin{equation*}
\renewcommand{\arraystretch}{1.8}
    \begin{array}{lcll}
    F_{\mathcal G, c}(\bar{\bm {x}}, \bar{\bm {\xi}}) & = & \displaystyle \sum_{n\in\mathcal N_c}\bm y_{n}(\bar{\bm x},\bar{\bm \xi}) & 
    \\
    & =  & \displaystyle \sum_{n\in\mathcal N_c} \mathbb I\left(\sum_{\nu\in\delta(n)}\bar{\bm\xi}_{\nu}\bar{\bm x}_{\nu} \geq 1\right) & 
    \\
    & \geq & W|\mathcal N_c| & 
    \\
    \end{array}
\end{equation*}

We define $\textstyle \bm y^{\star}_{n} = \mathbb I\left(\sum_{\nu\in\delta(n)}\bar{\bm\xi}_{\nu}\bar{\bm x}_{\nu} \geq 1\right)$ which is feasible in Problem~\eqref{prob:two_stage_final}.
Since the choice of $\bar{\bm \xi}$
was arbitrary, we showed that given a solution to Problem~\eqref{prob:single_stage_final}, we can always construct a feasible solution to Problem~\eqref{prob:two_stage_final}, thus the objective value of the latter is at least as high. 

We now prove the contrary, i.e., given a solution to Problem~\eqref{prob:two_stage_final}, we will construct a solution to Problem~\eqref{prob:single_stage_final}. Consider $\bar{\bm x}$ to be an optimal solution to Problem~\eqref{prob:single_stage_final}. Suppose there exists $\bar{\bm \xi}\in \Xi$ such that
\begin{equation*}
\renewcommand{\arraystretch}{1.8}
    \begin{array}{lcll}
    & F_{\mathcal G, c}(\bar{\bm {x}}, \bar{\bm {\xi}}) & < & \displaystyle {|\mathcal N_c|}W
    \\
    \Rightarrow &  \sum_{n\in\mathcal N_c} \mathbb I\left(\sum_{\nu\in\delta(n)}\bar{\bm\xi}_{\nu}\bar{\bm x}_{\nu} \geq 1\right) &  < & \displaystyle {|\mathcal N_c|}W.
    \end{array}
\end{equation*}
However, since $\bar{\bm x}$ is feasible in Problem~\eqref{prob:single_stage_final}, we have that
\begin{equation*}
\renewcommand{\arraystretch}{1.8}
    \begin{array}{lcl}
    & \forall \tilde{\bm \xi} \in \Xi, \; \exists \bm{y}_{n} :  & \; \bm{y}_{n} \leq \displaystyle \sum_{\nu\in\delta(n)}\tilde{\bm{\xi}_\nu} \bar{\bm {x}_\nu}
    \\
    & & \displaystyle \sum_{n\in\mathcal N_c}\bm y_{n} \geq |\mathcal N_c|W.
    \end{array}
\end{equation*}
By construction, $\textstyle \bm y_{n} \leq \mathbb{I}\left(\sum_{\nu \in \delta(n)}\tilde{\bm \xi_{\nu}}\bar{\bm x_{\nu}} \geq 1\right), \;\forall n\in\mathcal N .$ Thus
\begin{equation*}
\renewcommand{\arraystretch}{1.8}
    \begin{array}{lcl}
    \displaystyle \sum_{c\in\mathcal C}\sum_{n\in\mathcal N_c}\mathbb{I}\left(\sum_{\nu \in \delta(n)}\tilde{\bm \xi_{\nu}}\bar{\bm x_{\nu}} \geq 1 \right) & \geq  & \sum_{c\in\mathcal C}\sum_{n\in\mathcal N_c}\bm y_{n}
    \\
    &  \geq & |\mathcal N_c|W.
    \end{array}
\end{equation*}
According to the above result, we showed that the optimal objective value of Problem~\eqref{prob:single_stage_final} is at least as high as that of Problem~\eqref{prob:two_stage_final}. This completes the proof.
\end{proof}

\subsection{Exact MILP Formulation of the K-Adaptability Problem}\label{apd:MILP_formul}

In order to derive the equivalent MILP in Theorem~\ref{thm:main}, we start by a variant of the $K$-adaptability Problem~\eqref{prob:two_stage_k_adapt}, in which we move the constraints of the inner maximization problem to the definition of the uncertainty set in the spirit of~\cite{hanasusanto2015k}. Next, we prove, via Proposition~\ref{prop:relaxed}, that by relaxing the integrality constraint on the uncertain parameters~$\bm \xi$, the problem remains unchanged, and this is the key result that enables us to provide an \emph{equivalent} MILP reformulation for Problem~\eqref{prob:two_stage_k_adapt}.

\newar{We replace $\Xi$ with a collection of uncertainty sets parameterized by vectors $\bm \ell \in \mathcal L$ as in~\cite{hanasusanto2015k}. Specifically, it follows from Proposition 2 in~\cite{hanasusanto2015k} that Problem~\eqref{prob:two_stage_k_adapt} is equivalent to
\begin{equation}
\renewcommand{\arraystretch}{1.8}
\begin{array}{cl}
\displaystyle  \max & \displaystyle \min_{\lb \in \mathcal{L}} \quad \min_{\bm \xi \in \Xi(\bm x, \bm y, \lb)} \quad \max_{\begin{smallmatrix} k \in \mathcal K:\\ \bm \ell_{k}=0\end{smallmatrix}} \; \; \sum_{n\in \mathcal N} \bm y^{k}_{n} \\
\text{s.t.} & \bm x \in\mathcal X, \; \bm y^{1},\ldots,\bm y^{K} \in \mathcal Y,
\end{array}
\label{prob:two_stage_partitioned}
\end{equation}
where $\Xi(\bm x, \bm y, \lb)$ is defined through
\begin{equation*}
\displaystyle  \Xi(\bm x, \bm y, \lb) \; : = \; \left\{\bm \xi \in \Xi: 
\arraycolsep=1.4pt\def\arraystretch{2.2}
\begin{array}{cll} 
& \displaystyle \bm y^{k}_{\bm\ell_{k}} \; > \; \sum_{\nu \in \delta(\bm\ell_{k})}\bm\xi_{\nu}\bm x_{\nu}, \; \quad \qquad \; \forall k\in\mathcal K: \bm\ell_{k}>0\\
& \displaystyle \bm y^{k}_{n} \; \leq \; \sum_{\nu \in \delta(n)} \bm \xi_{\nu} \bm x_{\nu} \;\; \forall n\in \mathcal{N}, \; \; \forall k\in\mathcal K: \bm\ell_{k} = 0\\
\end{array}
\right\},
\end{equation*}
and, with a slight abuse of notation, we use $\bm y:=\{\bm y^1,\ldots,\bm y^{K}\}$. The vector $\lb \in \mathcal L$ encodes which of the $K$ candidate covering schemes are feasible. 
By introducing $\lb$, the constraints of the inner maximization problem are absorbed in the parameterized uncertainty sets $\Xi(\bm x, \bm y, \lb)$, and in the inner-most maximization problem, any covering scheme can be chosen for which $\lb_{k} = 0$. 

Note that, for any fixed ${\bm x} \in \mathcal X$, ${\bm y} \in \mathcal Y^K$, and ${\bm \ell} \in \mathcal L$, the strict inequalities in $\Xi(\bm x, \bm y, \lb)$ can be converted to (loose) inequalities as in
\begin{equation*}
\displaystyle  \Xi(\bm x, \bm y, \lb) = \left\{\bm \xi \in \Xi: 
\arraycolsep=1.4pt\def\arraystretch{2.2}
\begin{array}{cll} 
& \displaystyle \bm y^{k}_{\lb_{k}} \; \geq \; \sum_{\nu \in \delta(\lb_{k})}\bm \xi_{\nu}\bm x_{\nu} + 1, \; \quad \quad \; \forall k\in\mathcal K: \bm\ell_{k}>0\\
& \displaystyle {\bm y}^{k}_{n} \; \leq \; \sum_{\nu \in \delta(n)} \bm \xi_{\nu} \bm x_{\nu} \;\;\;\; \forall n\in \mathcal{N}, \; \; \forall k\in\mathcal K: \bm\ell_{k} = 0 \\
\end{array}
\right\}.
\end{equation*}
This idea was previously leveraged in~\cite{rahmattalabi2018robust}. 
It follows naturally since all decision variables and uncertain parameters are binary. Next, we show that we can obtain an equivalent problem by relaxing the integrality constraint on the set $\Xi$ in the definition of $\Xi({\bm x},{\bm y},{\bm l})$. Consider the following problem
\begin{equation}
\renewcommand{\arraystretch}{1.8}
\begin{array}{cl}
\displaystyle  \max &  \displaystyle  \min_{\lb \in \mathcal{L}} \quad  \min_{\bm \xi \in \overline\Xi(\bm x, \bm y, \lb)} \quad \max_{\begin{smallmatrix} k \in \mathcal K:\\ \lb_{k}=0\end{smallmatrix}} \; \;  \sum_{n\in \mathcal N} \bm{y}^{k}_{n} \\
\text{s.t.} & \bm x \in\mathcal X,\; \bm y \in \mathcal Y^{K}, 
\end{array}
\label{prob:two_stage_partitioned_relaxed_us}
\end{equation}
where the uncertainty set is obtained by relaxing the integrality constraints on ${\bm \xi}$, i.e.,
\begin{equation*}
\displaystyle  \overline\Xi(\bm x, \bm y, \lb) = \left\{\bm \xi \in \mathcal T: 
\arraycolsep=1.4pt\def\arraystretch{2.2}
\begin{array}{cll} 
& \displaystyle \bm y^{k}_{\lb_{k}} \; \geq \; \sum_{\nu \in \delta(\lb_{k})}\bm \xi_{\nu}\bm x_{\nu} + 1, \; \quad \quad \; \forall k\in\mathcal K: \bm\ell_{k}>0\\
& \displaystyle {\bm y}^{k}_{n} \; \leq \; \sum_{\nu \in \delta(n)} \bm \xi_{\nu} \bm x_{\nu} \;\;\;\; \forall n\in \mathcal{N}, \; \; \forall k\in\mathcal K: \bm\ell_{k} = 0 \\
\end{array}
\right\}.
\end{equation*}
}

\newar{
\begin{proposition}
Under Assumption~\ref{assumption:XiandY}, Problems~\eqref{prob:two_stage_partitioned} and~\eqref{prob:two_stage_partitioned_relaxed_us} are equivalent.
\label{prop:relaxed}
\end{proposition}
\begin{proof}
Let $\bm x \in \mathcal X$, $\bm y \in \mathcal Y^K$, and $\lb\in \mathcal L$. It suffices to show that
\begin{equation*}
\min_{\bm \xi \in \Xi(\bm x, \bm y, \lb)} \; \; \max_{\begin{smallmatrix} k \in \mathcal K:\\ \lb_{k}=0\end{smallmatrix}} \; \; \sum_{n\in \mathcal N} \bm{y}^{k}_{n}
\qquad \text{ and } \qquad 
\min_{\bm \xi \in \overline\Xi(\bm x, \bm y, \lb)} \; \; \max_{\begin{smallmatrix} k \in \mathcal K:\\ \lb_{k}=0\end{smallmatrix}} \; \; \sum_{n\in \mathcal N} \bm y^{k}_{n}
\end{equation*}
are equivalent. Observe that the these problems have the same objective function. Thus, the two problems have the same optimal objective value if and only if they are either both feasible or both infeasible. As a result, it suffices to show that $\Xi(\bm x, \bm y, \lb)$ is empty if and only if $\overline\Xi(\bm x, \bm y, \lb)$ is empty. Naturally, if $\overline\Xi(\bm x, \bm y, \lb)=\emptyset$ then $\Xi(\bm x, \bm y, \lb)=\emptyset$ since $\mathcal T$ is the linear programming relaxation of~$\Xi$. Thus, it suffices to show that the converse also holds, i.e., that if 
$\overline \Xi(\bm x, \bm y, \lb) \neq \emptyset$, then also $\Xi(\bm x, \bm y, \lb) \neq \emptyset$.

To this end, suppose that $\overline \Xi(\bm x, \bm y, \lb) \neq \emptyset$ and let $\tilde{\bm \xi} \in \overline \Xi(\bm x, \bm y, \lb)$. Then, $\tilde {\bm \xi}$ is such that
\begin{equation}
\renewcommand{\arraystretch}{1.8}
\begin{array}{l}
 \tilde {\bm \xi} \in \mathcal T, \\
 \displaystyle \bm y^{k}_{\lb_{k}} \; \geq \; \sum_{\nu \in \delta(\lb_{k})} \tilde{\bm \xi}_{\nu}\bm x_{\nu} + 1 \; \; \forall k \in \mathcal{K} : \lb_{k}>0, \\ 
 \displaystyle \bm y^{k}_{n} \; \leq \; \sum_{\nu \in \delta(n)} \tilde{\bm \xi}_{\nu}\bm x_{\nu} \; \; \forall n\in \mathcal N, \;  \forall k \in \mathcal{K} : \lb_{k}= 0. 
\end{array}
\label{eq:interm_implication0}
\end{equation}
Next, define $\hat{\bm\xi}_{n} := \lceil \tilde{\bm\xi}_{n} \rceil$ $\forall n\in \mathcal N$. We show that~$\hat{\bm \xi} \in \Xi(\bm x, \bm y,\lb)$. First, note that $\hat{\bm \xi} \geq \tilde{\bm \xi}$ and by Assumption~\ref{assumption:XiandY}, it follows that $\hat{\bm \xi} \in \mathcal T$. Moreover, by construction, $\hat{\bm \xi} \in \{0,1\}^N$. Thus, it follows that $\hat{\bm \xi}\in \Xi$. Next, we show that the constructed solution $\hat{\bm \xi}$ also satisfies the remaining constraints in $\Xi(\bm x, \bm y, \lb)$. Fix $k \in \mathcal K$ such that $\lb_{k} > 0$. Then, from~\eqref{eq:interm_implication0} it holds that
\begin{equation*}
\renewcommand{\arraystretch}{1.8}
\begin{array}{cll}
&&  \displaystyle \bm y^{k}_{\lb_{k}} \; \geq \; \sum_{\nu \in \delta(\lb_{k})} \tilde{\bm \xi}_{\nu}\bm x_{\nu} + 1 \\
 &\Rightarrow & \bm y_{\lb_{k}}^k=1 \; \text{ and } \; \tilde{\bm \xi}_{\nu}\bm x_{\nu} = 0 \quad \forall \nu \in \delta(\lb_{k})  \\
& \Rightarrow &  \bm y_{\lb_{k}}^k=1 \; \text{ and } \; \tilde{\bm \xi}_{\nu} = 0 \quad  \forall \nu \in \delta(\lb_{k}): \bm x_{\nu} = 1 \\
& \Rightarrow &  \bm y_{\lb_{k}}^k=1 \; \text{ and } \; \hat{\bm \xi}_{\nu} = 0 \quad \forall \nu \in \delta(\lb_{k}): \bm x_{\nu} = 1  \\
& \Rightarrow & \bm{y}^{k}_{\lb_{k}} \; \geq \; \sum_{\nu \in \delta(\lb_{k})} \hat{\bm \xi}_{\nu}\bm{x}_{\nu} + 1,
\end{array}
\end{equation*}
where the first and second implication follow since ${\bm y}$ and ${\bm x}$ are binary, respectively, and the third implication holds by definition of $\hat{\bm \xi}$, 

Next, fix $k \in \mathcal K$ such that $\lb_{k} = 0$. Then,~\eqref{eq:interm_implication0} yields
\begin{equation*}
\renewcommand{\arraystretch}{1.8}
\begin{array}{cll}
&&  \bm y^{k}_{n} \leq \sum_{\nu \in \delta(n)} \tilde{\bm{\xi}}_{\nu} \bm{x}_{\nu} \;\; \forall n \in \mathcal{N}  \\
& \Rightarrow & \bm{y}^{k}_{n} \leq \sum_{\nu \in \delta(n)} \hat{\bm\xi}_{\nu} \bm{x}_{\nu} \;\; \forall n \in \mathcal{N},
\end{array}
\end{equation*}
which follows by definition of $\hat{\bm \xi}$. We have thus constructed $\hat{\bm \xi} \in \Xi(\bm x, \bm y, \lb)$ and therefore conclude that $\Xi(\bm x, \bm y, \lb)\neq \emptyset$. Since the choice of $\bm x \in \mathcal X$, $\bm y \in \mathcal Y^K$, and $\lb\in \mathcal L$ was arbitrary, the claim follows.
\end{proof}

Proposition~\ref{prop:relaxed} is key to leverage existing literature to reformulate Problem~\eqref{prob:two_stage_k_adapt} as an MILP. The reformulation is based on~\cite{hanasusanto2015k,rahmattalabi2018robust}. 
}

\begin{proof}[Proof of Theorem~\ref{thm:main}]


Note that the objective function of the Problem~\eqref{prob:two_stage_partitioned} is identical to
\begin{equation*}
\renewcommand{\arraystretch}{1.8}
\begin{array}{cll}
\displaystyle \min_{\lb \in \mathcal{L}} & \min_{\bm \xi \in \overline{\Xi}(\bm x, \bm y ,\lb)} & \left[\max_{\bm \lambda \in \Delta_{K}(\lb)} \; \sum_{k\in \mathcal{K}}\bm \lambda_{k}\sum_{n\in \mathcal N} \bm y^{k}_{n} \right],
\end{array}
\end{equation*}
where $\Delta_{K}(\lb) := \{\bm{\lambda} \in \mathbb{R}^{K}_{+}: \textbf{e}^\top \bm\lambda  = 1, \; \bm\lambda_{k} = 0 \;\; \forall k\in \mathcal K : \lb_{k} \neq 0\}.$ 
We define 
$\partial \mathcal{L} := \{ \lb \in \mathcal{L}: \lb \ngtr \bm 0 \}$, and  $\mathcal{L}_{+} := \{\lb \in \mathcal{L} : {\bm \ell} > \bm 0\}$. 
We remark that $\Delta_{K}(\bm \ell) = \emptyset$ if and only if
$\lb > \bm 0$. If~$\overline{\Xi}(\bm x, \bm y, \lb) =\emptyset$ for all $\lb \in \mathcal{L}_{+}$, then the problem is equivalent to
\begin{equation*}
\renewcommand{\arraystretch}{1.8}
\begin{array}{cll}
\displaystyle \min_{\lb \in \partial\mathcal{L}} & \displaystyle \min_{\bm \xi \in \overline{\Xi}(\bm x, \bm y ,\lb)} & \displaystyle \left[\max_{\bm \lambda \in \Delta_{K}(\lb)} \; \sum_{k\in \mathcal{K}}\bm \lambda_{k}\sum_{n\in \mathcal N} \bm y^{k}_{n} \right].
\end{array}
\end{equation*}
By applying the classical min-max theorem, we obtain
\begin{equation*}
\renewcommand{\arraystretch}{1.8}
\begin{array}{cll}
\displaystyle \min_{\lb \in \partial\mathcal{L}} & \displaystyle \max_{\bm \lambda \in \Delta_{K}(\lb)} & \displaystyle\min_{\bm \xi \in \overline{\Xi}(\bm x, \bm y ,\lb)} \; \sum_{k\in \mathcal{K}}\bm \lambda_{k}\sum_{n\in \mathcal N} \bm y^{k}_{n}.
\end{array}
\end{equation*}
This problem is also equivalent to
\begin{equation*}
\renewcommand{\arraystretch}{1.6}
\begin{array}{cll}
\displaystyle \max_{\bm \lambda(\lb) \in \Delta_{K}(\lb)} & \displaystyle\min_{\lb \in \partial\mathcal{L}} & \displaystyle \min_{\bm \xi \in \overline{\Xi}(\bm x, \bm y , \lb)} \; \sum_{k\in \mathcal{K}}\bm \lambda_{k}(\lb)\sum_{n\in \mathcal N} \bm y^{k}_{n}.
\end{array}
\end{equation*}
If on the other hand $\overline{\Xi}(\bm x, \bm y ,\lb) \neq \emptyset$ for some $\lb \in \mathcal{L}_{+}$, the objective of Problem~\eqref{prob:two_stage_partitioned} evaluates to $-\infty$.

Using the above results, we can write Problem~\eqref{prob:two_stage_partitioned} in epigraph form as 
\begin{equation}
\renewcommand{\arraystretch}{1.6}
\begin{array}{cll}
\max \displaystyle & \tau & \\
\text{s.t.} & {\bm x \in \mathcal{X}}, \; \bm y \in \mathcal Y^{K},\; \tau \in \mathbb{R}, \; \bm \lambda(\lb) \in \Delta_{K}(\lb), \; \lb \in \partial \mathcal{L} &
\\
& \tau \; \leq \; \displaystyle \sum_{k \in \mathcal{K}} \bm\lambda_{k}(\lb)\sum_{n\in \mathcal N} \bm y^{k}_{n} \quad \forall \lb \in \partial \mathcal{L} : \; \Xi(\bm x, \bm y , \lb) \neq \emptyset &\\
&\overline{\Xi}(\bm x, \bm y, \lb) = \emptyset \quad \forall \lb \in \mathcal{L}_{+}.&
\end{array}
\label{prob:two_stage_epi}
\end{equation}
We begin by reformulating the semi-infinite constraint associated with $\lb\in\partial \mathcal L$ in Problem~\eqref{prob:two_stage_epi}. To this end, fix $\lb\in\partial \mathcal L$ and consider the linear program
\begin{equation*}
\renewcommand{\arraystretch}{1.8}
\begin{array}{cll}
\min & 0\\
\text{s.t.} &  0 \; \leq \; \bm \xi_{n} \leq 1 \;\; \forall n \in \mathcal N\\
& \bm{A}^{\top} \bm \xi \; \geq \;  \bm b\\
& \displaystyle \bm y^{k}_{\lb_{k}} \; \geq \; \sum_{\nu \in \delta(\lb_{k})}\bm \xi_{\nu}\bm x_{\nu} + 1 \quad \forall k \in \mathcal{K} \; : \; \lb_{k}>0 \\
& \displaystyle \bm y^{k}_{n} \; \leq \; \sum_{\nu \in \delta(n)} \bm \xi_{\nu}\bm x_{\nu} \quad \forall n\in \mathcal N, \; \forall k \in \mathcal{K} \; : \; \lb_{k} = 0,
\end{array}
\end{equation*}
whose dual reads
\begin{equation*}
\renewcommand{\arraystretch}{1.6}
\begin{array}{cll}
\max & \displaystyle  -\textbf{e}^{\top} \bm\theta(\lb)  + \bm{b}^{\top}\bm{\alpha}(\lb) - \sum_{\begin{smallmatrix}
k \in\mathcal {K}\\
\lb_{k} \neq 0
\end{smallmatrix}} \left(\bm y^{k}_{\lb_{k}} -1 \right)\bm \nu_{k}(\lb)+ \sum_{\begin{smallmatrix}
k\in\mathcal {K} \\
\lb_{k} = 0
\end{smallmatrix}
}
\sum_{n\in\mathcal N}{\bm y^{k}_{n}}\bm \beta_{n}^{k}(\lb) & \\
\text{s.t.} & \bm{\theta}(\lb) \in \mathbb{R}^{N}_{+}, \; \bm{\alpha}(\lb) \in \mathbb{R}^{R}_{+}, \; \bm \beta^{k}(\lb) \in \mathbb{R}^{N}_{+},\; \forall k \in \mathcal{K}, \; \bm \nu(\lb) \in \mathbb{R}^{K}_{+} & \\
& 
\displaystyle \bm \theta_{n}(\lb)  \; \leq \; \bm A^{\top}\bm{\alpha}(\lb) +  \sum_{\begin{smallmatrix}
k\in\mathcal {K}\\
\lb_{k} \neq 0
\end{smallmatrix}} 
\sum_{\nu\in\delta(\lb_{k})}{\bm x_{\nu}}\bm \nu_{k}(\lb) -  \displaystyle\sum_{\begin{smallmatrix}
k \in \mathcal {K}\\
\lb_{k} = 0
\end{smallmatrix}}\sum_{ \nu\in\delta(n)}{\bm x_{\nu}}\bm{\beta}_{n}^{k}(\lb)  \quad \forall n \in \mathcal N.
\end{array}
\end{equation*}
In Problem~\eqref{prob:two_stage_epi} the constraint associated with each $\bm \ell \in \partial \mathcal{L}$ is satisfied if and only if the objective value of the above dual problem is greater than $\tau - \textstyle \sum_{k\in\mathcal{K}}\bm \lambda_{k}(\lb)\sum_{n\in \mathcal N} \bm y^{k}_{n}$. This follows since the dual is always feasible. Therefore, either the dual is unbounded in which case the primal is infeasible, i.e., $\Xi(\bm x, \bm y, \lb) = \emptyset$, and the constraint is trivial. Else, by strong duality, the primal and dual must have the same objective value (zero). As a result, the constraints in Problem~\eqref{prob:two_stage_epi} associated with each $\bm \ell \in \partial \mathcal{L}$ can be written as
\begin{equation*}
\renewcommand{\arraystretch}{1.8}
\begin{array}{cll}
& \displaystyle \tau \leq - \textbf{{{e}}}^{\top}\bm\theta(\lb) + \bm{b}^{\top}\bm{\alpha}(\lb) - \sum_{\begin{smallmatrix}
k\in\mathcal {K}\\
\lb_{k} \neq 0
\end{smallmatrix}}\left(\bm{y}^{k}_{\lb_{k}} -1\right)\bm \nu_{k}(\lb)+ \sum_{\begin{smallmatrix}
k\in\mathcal {K} \\
\lb_{k} = 0
\end{smallmatrix}
}
\sum_{n\in\mathcal N}{\bm{y}^{k}_{n}}\bm\beta_{n}^{k}(\lb)+ \sum_{k\in\mathcal{K}}\bm\lambda_{k}(\lb)\sum_{n\in \mathcal N} \bm{y}^{k}_{n} & \\
& 
\displaystyle \bm \theta_{n}(\lb)  \; \leq \; \bm A^{\top}\bm{\alpha}(\lb) +  \sum_{\begin{smallmatrix}
k\in\mathcal {K}\\
\lb_{k} \neq 0
\end{smallmatrix}} 
\sum_{\nu\in\delta(\lb_{k})}{\bm x_{\nu}}\bm \nu_{k}(\lb) -  \displaystyle\sum_{\begin{smallmatrix}
k \in \mathcal {K}\\
\lb_{k} = 0
\end{smallmatrix}}\sum_{ \nu\in\delta(n)}{\bm x_{\nu}}\bm{\beta}_{n}^{k}(\lb)  \quad \forall n \in \mathcal N.
\end{array}
\end{equation*}

Finally, the last constraint in Problem~\eqref{prob:two_stage_epi} is satisfied if the linear program
\begin{equation*}
\renewcommand{\arraystretch}{1.8}
\begin{array}{cll}
\min & \; \displaystyle 0 & \\
\text{s.t.} & \; 0  \; \leq \; \bm \xi_{n} \leq 1 \quad \forall n\in \mathcal N \\
& \bm A \bm \xi \; \geq \; \bm{b} \\
& \; \displaystyle \bm y^{k}_{\lb_{k}} \; \geq \; \sum_{\nu \in \delta(\lb_{k})}\bm \xi_{\nu}\bm x_{\nu} + 1 \quad \forall k \in \mathcal{K} \; : \; \lb_{k} \neq 0
\end{array}
\end{equation*}
is infeasible. Using strong duality, this occurs if the dual problem
\begin{equation*}
\renewcommand{\arraystretch}{1.8}
\begin{array}{cll}
\max & \displaystyle 
- \textbf{{{e}}}^{\top}\bm\theta(\bm l) + \bm{\alpha}(\lb)^{\top}\bm{b}  - \sum_{\begin{smallmatrix}
k\in\mathcal {K}\\
\lb_{k} \neq 0
\end{smallmatrix}}\left(\bm y^{k}_{\lb_{k}} - 1 \right)\bm \nu_{k}(\lb)\\
\text{s.t.} & \bm{\theta}(\lb)\in \mathbb{R}_{+}^{N},\; {\bm \alpha}(\lb) \in \mathbb{R}^{R}_{+}, \; \bm \nu(\lb) \in \mathbb{R}_{+}^{K}\\ & 
\bm \theta_{n}(\lb)  \; \leq \; \bm A^{\top}\bm{\alpha}(\lb) + \displaystyle \sum_{\begin{smallmatrix}
k\in\mathcal {K}\\
\lb_{k} \neq 0
\end{smallmatrix}}\sum_{\nu\in\delta(\lb_{k})}{\bm x_{\nu}}\bm \nu_{k}(\lb) 
\quad \forall n \in \mathcal{N}
\end{array}
\end{equation*}
is unbounded. Since the feasible region of the dual problem constitutes a cone, the dual problem is
unbounded if and only if there is a feasible solution with an objective value of 1 or more. \end{proof}

\section{Supplemental Material: Bender's Decomposition}
\label{sec:Benders-problem-master-sub}

We do not detail all the steps of the Bender's decomposition algorithm. We merely provide the initial relaxed master problem and the subproblems used to generate the cuts. We refer the reader to e.g.,~\cite{bertsimas1997introduction} for more details. 

\textbf{Relaxed Master Problem.}
Initially, the relaxed master problem only involves the binary variables of the Problem~\eqref{prob:K-adaptability_MILP} and is expressible as
\begin{equation*}
\max \; \left\{  \tau \; : \; \tau \in \mathbb R, \; \bm x \in \mathcal X, \; \bm y^{1},\ldots,\bm y^{K} \in \mathcal Y \right\}.
\end{equation*}


\textbf{Subproblems. } As discussed in Section~\ref{sec:solution-approach}, Problem~\eqref{prob:K-adaptability_MILP} decomposes by ${\bm \ell}$. Depending on the index~${\bm \ell}$ of the subproblem, there are two types of subproblems to consider. If $\lb \in \mathcal L_0$, the subproblem is given by
%
\begin{equation*}
\renewcommand{\arraystretch}{1.6}
\begin{array}{cl}
\tag{$\mathcal Z_{0}(\lb)$}
\min & 0  \\
\text{s.t.} & \bm{\theta}(\lb),\; \bm \beta^{k}(\lb) \in \mathbb{R}^{N}_{+},\;  \bm{\alpha}(\lb) \in \mathbb{R}^{R}_{+},\;  \bm\nu(\lb) \in \mathbb{R}^{K}_{+},\; \bm \lambda(\lb) \in \Delta_{K}(\lb) 
\\
& \displaystyle \tau \; \leq \;  - \displaystyle {\textbf{{e}}}^{\top} {\bm\theta}(\bm \ell) + {\bm b}^{\top} {\bm \alpha}({\bm \ell}) - \sum_{
\begin{smallmatrix} k\in\mathcal K : \\ {\bm \ell}_{k} \neq 0 \end{smallmatrix}} \left( {\bm y}^{k}_{{\bm \ell}_k} -1 \right) {\bm \nu}_k(\bm \ell) + \ldots  \\
& \qquad \qquad \qquad \ldots + \sum_{ \begin{smallmatrix} k \in \mathcal K : \\ {\bm \ell}_k = 0 \end{smallmatrix} } \sum_{n\in\mathcal N}{\bm y}^k_n {\bm \beta}_n^k({\bm \ell}) + \sum_{k\in\mathcal K} {\bm \lambda}_k({\bm \ell}) \sum_{n\in \mathcal N} {\bm y}^k_n \\
%
\\
& \bm \theta_{n}(\lb) \; \leq \; \bm A^{\top}\bm{\alpha}(\lb) + \displaystyle\sum_{\begin{smallmatrix}
k\in\mathcal {K}\\
\lb_{k}\neq 0
\end{smallmatrix}}\sum_{\nu\in\delta(l_{k})}{\bm x_{\nu}}\bm \nu_{k}(\lb) - \displaystyle \sum_{\begin{smallmatrix}
k\in\mathcal {K}\\
\lb_{k} = 0
\end{smallmatrix}}\sum_{\nu\in\delta(n)}{\bm x_{\nu}}\bm \beta_{n}^{k}(\lb) \;\; \forall n \in \mathcal N. 
\end{array}
\end{equation*}
%
In a similar fashion, we define the subproblem associated with $\lb \in \mathcal L_{+}$, given by
\begin{equation*}
\renewcommand{\arraystretch}{1.6}
\begin{array}{cl}
\tag{$\mathcal Z_{+}(\lb)$}
\min & 0 \\
\text{s.t.} & \bm{\theta}(\lb)\in \mathbb{R}_{+}^{N},\; \bm{\alpha}(\bm l) \in \mathbb{R}_{+}^{R}, \; \bm \nu(\bm l) \in \mathbb{R}_{+}^{K} \\
& 1 \leq -\textbf{{{e}}}^{\top}\bm \theta(\bm l) + \bm {b}^{\top}\bm{\alpha}(\lb) - \sum_{\begin{smallmatrix}
k\in\mathcal {K}\\
\lb_{k} \neq 0
\end{smallmatrix}} \left(\bm y^{k}_{\lb_{k}} - 1 \right) \bm \nu_{k}(\lb) \\
& \bm \theta_{n}(\lb) \leq  \bm A^{\top}\bm{\alpha}(\lb) + \displaystyle \sum_{\begin{smallmatrix}
k\in\mathcal {K}     \\
\lb_{k} \neq 0
\end{smallmatrix}}
\sum_{\nu\in\delta(\lb_{k})}{\bm x_{\nu}}\bm \nu_{k}(\lb)  \quad \forall n \in \mathcal N. 
\end{array}
\end{equation*}

\end{document}